\documentclass[twoside,11pt,reqno]{amsart} 
\usepackage[top=1in,bottom=1in,left=1.25in,right=1.25in]{geometry}
\usepackage{amsmath} \usepackage{amssymb} \usepackage{color}
\usepackage{amsfonts} \usepackage{setspace} \usepackage{enumerate}
\usepackage[toc,page]{appendix}
\usepackage[top=1in,bottom=1in,left=1.25in,right=1.25in]{geometry}

\newtheorem{theorem}{Theorem}[section]
\newtheorem{proposition}[theorem]{Proposition}
\newtheorem{lemma}[theorem]{Lemma}
\newtheorem{corollary}[theorem]{Corollary}
 \theoremstyle{remark}
\newtheorem{remark}[theorem]{Remark} 
\theoremstyle{definition} \newtheorem{definition}[theorem]{Definition}

\numberwithin{equation}{section}

\newcommand{\R}{\mathbb{R}} \newcommand{\la}{\lambda}
\newcommand{\vp}{\varphi} \newcommand{\C}{\mathbb{C}}
\newcommand{\Hb}{\overline{\mathbb{H}}} \newcommand{\mb}{\mathbf}
\newcommand{\ve}{\varepsilon} \newcommand{\mc}{\mathcal}
\renewcommand{\Re}{\mathrm{Re}\,} 
\newcommand{\rg}{\mathrm{rg}\,} \newcommand{\rank}{\mathrm{rank}\,}
\newcommand{\B}{\mathbb{B}}

\title[Blowup for wave maps into a negatively curved target]{On the
  existence and stability of blowup for wave maps into a negatively
  curved target}

\author{Roland Donninger} \address{Rheinische
  Friedrich-Wilhelms-Universit\"at Bonn, Mathematisches Institut,
  Endenicher Allee 60, D-53115 Bonn, Germany} 
\email{donninge@math.uni-bonn.de}

\address{Universit\"at
  Wien, Fakult\"at f\"ur Mathematik, Oskar-Morgenstern-Platz 1, A-1090
  Vienna, Austria} 
\email{roland.donninger@univie.ac.at}

\author{Irfan Glogi\'c} \address{Department of Mathematics, The Ohio
  State University, 231 W 18th Ave, Columbus, OH, 43220, USA}
\email{glogic.1@osu.edu}

\usepackage{kantlipsum}
\usepackage[colorlinks=true,linkcolor=red,citecolor=red,urlcolor=blue]{hyperref}

\begin{document}

\begin{abstract}
  We consider wave maps on $(1+d)$-dimensional Minkowski space. For
  each dimension $d\geq 8$ we construct a negatively curved,
  $d$-dimensional target manifold that allows for the existence of a
  self-similar wave map which provides a stable blowup mechanism for
  the corresponding Cauchy problem. 
\end{abstract}

\maketitle

\section{Introduction}
\noindent We consider the Cauchy problem for a wave map from the Minkowski
spacetime $(\R^{1,d},\eta)$ into a warped product manifold
$N^d=\R^+ \times_g\, \mathbb{S}^{d-1}$ with metric $h$, see
e.g.~\cite{ONe83, Tac85} for a definition.  The metric $h$ has the form
\begin{equation}\label{Def:Metric}
  h=du^2+g(u)^2d\theta^2
\end{equation}
where $(u,\theta)\in \R^+\times \mathbb{S}^{d-1}$ is the natural polar coordinate
system on $N^d$, $d\theta^2$ is the standard metric on $\mathbb{S}^{d-1}$
and 
\begin{equation}
  \label{eq:assumg}
  g\in C^\infty(\R),\quad g\mbox{ is odd},\quad g'(0)=1,\quad
 g>0 \mbox{ on }(0,\infty).
\end{equation}
Furthermore, we endow
the Minkowski space with standard spherical coordinates
$(t,r,\omega)\in\R\times\R^+\times\mathbb{S}^{d-1}$. The metric $\eta$
thereby becomes
\begin{equation}\label{Eq:MetricRotMink}
  \eta = - dt^2 + dr^2 + r^2d\omega^2.
\end{equation}
In this setting, a map $U:(\R^{1,d},\eta)\rightarrow (N^d,h)$ can
be written as
\[
U(t,r,\omega)=(u(t,r,\omega),\theta(t,r,\omega)).
\]
We restrict our attention to the special subclass of so-called
1-equivariant or corotational maps where
\begin{equation*}
  u(t,r,\omega)=u(t,r) \quad\text{and}\quad  \theta(t,r,\omega)=\omega.
\end{equation*}
Under this ansatz the wave maps equation for $U$ reduces to the single
semilinear radial wave equation
\begin{equation}\label{Eq:CorWMgen}
  \left (\partial_t^2-\partial_r^2-\frac{d-1}{r}\partial_r\right )u(t,r)+\frac{d-1}{r^2}g(u(t,r))g'(u(t,r))=0,
\end{equation}
see e.g.~\cite{ShaTah94}.

It is not hard to see that the Cauchy problem for
Eq.~\eqref{Eq:CorWMgen} is locally well-posed for sufficiently smooth
data and even the low-regularity theory is well understood
\cite{ShaTah94}.
Consequently, the interesting questions concern the global Cauchy
problem and in particular, the formation of singularities in finite
time.
There is by now a sizable literature on blowup for wave maps which we
cannot review here in its entirety. Let it suffice to say that the
energy-critical case $d=2$ attracted particular attention, see
e.g.~\cite{BizChmTab01, Str03, KriSchTat08, RodSte10, RapRod12,
  SteTat10a, SteTat10b, KriSch12, CotKenLawSch15a, CotKenLawSch15b,
  Cot15, CanKri15, LawOh16} for recent contributions. In
supercritical dimensions $d\geq 3$ the existence of self-similar
solutions is typical \cite{Sha88, TurSpe90, CST98, Biz00, BizonBiernat15} and
stability results for blowup were obtained in \cite{BizChmTab00, Donninger11,
  DonSchAic12, BizonBiernat15, BieBizMal16, ChaDonnGlo17}. For nonexistence of type
II blowup see \cite{DodLaw15}. Note, however, that there exists
nonself-similar blowup in sufficiently high dimensions \cite{GhoIbrNgu17}.

According to a heuristic principle, one typically has finite-time blowup
if the curvature of the target is positive.
For negatively curved targets, on the other hand,
one expects global well-posedness.
A notable exception to that rule is provided by the construction of a self-similar solution for
a negatively curved target for $d=7$ in \cite{CST98}, which indicates
that the situation is more subtle.
In the present paper we show that the example from \cite{CST98} is not
a peculiarity. We construct suitable target manifolds for any
dimension $d\geq 8$ that allow for the existence of an \emph{explicit}
self-similar
solution. Moreover, we claim that the corresponding self-similar
blowup is nonlinearly asymptotically stable under small perturbations
of the initial data. In the case $d=9$ we prove this claim
rigorously. This provides the first example of \emph{stable} blowup
for wave maps into a negatively curved target.

\subsection{Self-similar solutions}
In order to look for self-similar solutions, we first observe that Eq.~\eqref{Eq:CorWMgen}
has the natural scaling symmetry
\begin{equation}\label{Eq:NatScaling}
  u(t,r) \mapsto
  u_{\lambda}(t,r):=u\left(\tfrac{t}{\la},\tfrac{r}{\la}\right),\quad \lambda>0,
\end{equation}
in the sense that if $u$ solves Eq.~\eqref{Eq:CorWMgen} then $u_{\la}$
solves it, too. Consequently, it is natural to look for solutions of
the form $u(t,r)=\phi(\frac{r}{t})$. Taking into account the time
translation and reflection symmetries of Eq.~\eqref{Eq:CorWMgen},
we arrive at the slightly more general ansatz
\begin{equation}\label{Eq:SimilarityAnsatz}
  u(t,r)=\phi(\rho), \quad \rho=\frac{r}{T-t},
\end{equation}
where the free parameter $T>0$ is the blowup time. By plugging the
ansatz~\eqref{Eq:SimilarityAnsatz} into Eq.~\eqref{Eq:CorWMgen} we
obtain the ordinary differential equation
\begin{equation}\label{Eq:CorWMode}
  (1-\rho^2)\phi''(\rho)+\left(\frac{d-1}{\rho}-2\rho\right)\phi'(\rho)-\frac{(d-1)g(\phi(\rho))g'(\phi(\rho))}{\rho^2}=0.
\end{equation}
By recasting~\eqref{Eq:CorWMode} into an integral equation and then using a fixed point argument one can show that any solution to Eq.~\eqref{Eq:CorWMode} that vanishes together with its first derivative at $\rho=0$ is identically zero near $\rho=0$. Therefore, any nontrivial smooth solution $\phi$ to
Eq.~\eqref{Eq:CorWMode} for which $\phi(0)=0$ must have $\phi'(0)\not=0$, and since 
\begin{equation}
  \frac{\partial}{\partial r}\phi\left(\frac{r}{T-t}\right)\bigg|_{r=0}=\frac{\phi'(0)}{T-t},
\end{equation}
such $\phi$ gives rise to a
smooth solution of Eq.~\eqref{Eq:CorWMgen} which suffers a gradient
blowup at the origin in finite time. Furthermore, due to finite
speed of propagation, this type of singularity arises from smooth,
compactly supported initial data. 
In the following, we restrict ourselves to the study of the solution 
in the backward lightcone of the singularity,
\begin{equation}\label{Eq:LightCone}
  \mathcal{C}_{T} := \{ (t,r): t\in [0,T),  r \in[0,T - t]   \}.
\end{equation}
Note that in terms of the coordinate $\rho$, $\mathcal{C}_T$ corresponds to the interval
$[0,1]$. Consequently, we look for solutions of
Eq.~\eqref{Eq:CorWMode} that belong to $C^{\infty}[0,1].$

\section{Existence of blowup for a negatively curved target manifold}
\noindent In this section we construct for every $d\geq 8$ a negatively curved
$d-$dimensional Riemannian manifold $(N^d,h)$
which allows for a wave map $U:(\R^{1,d},\eta)\rightarrow (N^d,h)$
that starts off smooth and blows up in finite time. We do this by a
suitable choice of the function $g$ that defines the metric on $N^d$ by
means of~\eqref{Def:Metric}. To begin with, we restrict ourselves to
small $u$ and set
\begin{equation}
\label{Def:g}
 g(u):=u\sqrt{1+7u^2-(23d-170)u^4}. 
\end{equation}
Clearly, $g$ is odd and smooth locally around the origin. Furthermore,
$g(u)>0$ for small $u>0$ and $g'(0)=1$, cf.~\eqref{eq:assumg}.
In addition, for $d\geq8$, the
metric~\eqref{Def:Metric} makes the manifold $N^d$ negatively curved
locally around $u=0$, see Proposition~\ref{Prop:NegCurv}. Next,
Eq.~\eqref{Eq:CorWMgen} takes the form
\textsl{}\begin{equation}\label{Eq:CorWM_g}
    \left (\partial_t^2-\partial_r^2-\frac{d-1}{r}\partial_r\right
    )u(t,r)+
\frac{(d-1)[u(t,r)+14u(t,r)^3-3(23d-170)u(t,r)^5]}{r^2}=0,
\end{equation}
and the corresponding ordinary differential
equation~\eqref{Eq:CorWMode} becomes
\begin{equation}\label{Eq:CorWMode_g}
  (1-\rho^2)\phi''(\rho)
  +\left(\frac{d-1}{\rho}-2\rho\right)\phi'(\rho)
-\frac{(d-1)[\phi(\rho)+14\phi(\rho)^3-3(23d-170)\phi(\rho)^5]}{\rho^2}=0.
\end{equation}
As already discussed, any nonzero function $\phi\in C^{\infty}[0,1]$ that
solves Eq.~\eqref{Eq:CorWMode_g} and vanishes at $\rho=0$ yields a classical solution to
Eq.~\eqref{Eq:CorWM_g} that blows up in finite time. In fact,
Eq.~\eqref{Eq:CorWMode_g} has an \emph{explicit} formal solution
\begin{equation}\label{Eq:Sol}
  \phi_0(\rho)=\frac{a\rho}{\sqrt{b-\rho^2}}
\end{equation}
where
\begin{equation*}
  a=\sqrt{\frac{d}{E(d)}},\quad \quad  b=1+\frac{d}{2}-\frac{7d(d-1)}{E(d)}
\end{equation*}
and \[ E(d)=\sqrt{(46d^2-291d-49)(d-1)}+7(d-1). \]
Furthermore, if $d\geq8$ then $E(d)$ is positive and $b>1$, which makes
$\phi_0$ a smooth and increasing function on $[0,1]$.  Now we
have the following result.
\begin{proposition}\label{Prop:NegCurv}
  For each $d\geq8$ there exists an $\ve>0$ and a function
  $g: \R\to\R$ satisfying \eqref{eq:assumg} such that
  $g(u)= u\sqrt{1+7u^2-(23d-170)u^4}$ for $|u|<\phi_0(1)+\ve$
  and the manifold $N^d$ with metric given by~\eqref{Def:Metric}
 has all sectional curvatures negative.
\end{proposition}
\noindent The proof is somewhat lengthy but elementary and therefore
postponed to the Appendix, see Sec.~\ref{Sec:ProofNegCurv}.

We define
\begin{equation}\label{eq:u^T}
  u^T(t,r):=\phi_0\left(\frac{r}{T-t}\right), \quad (t,r)\in\mathcal{C}_T.
\end{equation}
Note that $|\phi_0(\rho)|\leq \phi_0(1)$ for all $\rho\in [0,1]$ and
thus, 
\[ U^T(t,r,\omega):=(u^T(t,r),\omega) \]
is a wave map from $\mc C_T \subset \R^{1,d}$ to $(N^d,h)$.
By finite speed of propagation we obtain the following result.

\begin{theorem}
  For every $d\geq8$ there exists a $d$-dimensional, negatively curved
  Riemannian manifold $N^d$ such that the Cauchy problem for wave
  maps from Minkowski space $\R^{1,d}$ into $N^d$ admits a
  solution which develops from smooth Cauchy data of
  compact support and forms a singularity in finite time.
\end{theorem}
\begin{remark}
  Any function $g$ of the form
  \begin{equation}\label{Eq:MetricGeneral}
    g(u)^2=u^2+c_1u^4+c_2u^6, \quad c_1,c_2\in\R
  \end{equation} gives rise to a (formal) solution to
  Eq.~\eqref{Eq:CorWMode} of the form~\eqref{Eq:Sol}. However, for $d
  \leq 7$ there is provably no choice of real numbers $c_1,c_2$
  in~\eqref{Eq:MetricGeneral} such that the corresponding
  solution~\eqref{Eq:Sol} is smooth on $[0,1]$ and $u^T$ from~\eqref{eq:u^T} stays inside the negatively curved
  neighborhood of the pole $u=0$ whose metric is given by~\eqref{Eq:MetricGeneral}. 
\end{remark}

In order to determine the role of the solution $u^T$ for generic
evolutions, it is necessary to investigate its stability under
perturbations.
In fact,  
we claim that for any $d\geq 8$, the self-similar
solution \eqref{eq:u^T} exhibits \emph{stable} blowup, i.e., there is
an open set of radial initial data that give rise to solutions which
approach $u^T$ in $\mathcal{C}_T$ as $t\rightarrow T^-$. The rest of
the paper is devoted to the proof of this stability property. For
simplicity we restrict ourselves to the lowest
odd dimension $d=9$.

\section{Stability of blowup}
\noindent From now on we fix $d=9$. In view of Eqs.~\eqref{Eq:CorWMgen} and
\eqref{Def:g}, we consider the Cauchy problem
\begin{align}\label{Eq:CorWMd9}
  \begin{cases}
    \displaystyle{\left
        (\partial_t^2-\partial_r^2-\frac{8}{r}\partial_r \right
      )u(t,r)
      +\frac{8[u(t,r)+14u(t,r)^3-111u(t,r)^5]}{r^2}=0,}\quad &(t,r) \in \mathcal{C}_T\\
    u(0,r)=u_0(r),\quad \partial_0 u(0,r)=u_1(r) &r\in[0,T].
  \end{cases}
\end{align}
The restriction to the backward lightcone $\mc C_T$ is possible and
natural by finite speed of propagation.
Furthermore, to
ensure regularity of the solution at the origin $r=0$, we
impose the boundary condition
\begin{equation}\label{Eq:BoundaryCondition}
  u(t,0)=0 \quad \text{for} \quad t\in [0,T).
\end{equation}
The blowup solution~\eqref{Eq:Sol} now becomes
\begin{equation}\label{eq:sol9}
  u^T(t,r)=\phi_0(\rho)=\frac{3\rho}{\sqrt{2(155-74\rho^2)}} \quad\text{where}\quad \rho=\frac{r}{T-t}.\newline
\end{equation}
Note that by construction, the wave map evolution for the target manifold $N^9$ is
given by Eq.~\eqref{Eq:CorWMd9}, provided that $|u(t,r)|\leq\phi_0(1)+\ve_1$
for some small $\ve_1>0$. We are only interested in the evolution
in the backward lightcone of the point of blowup and therefore study
Eq.~\eqref{Eq:CorWMd9} with no a priori restriction on the size of $u$.
A posteriori we show that the solutions we construct stay below $\phi_0(1)+\epsilon_1$.

Note further that Eq.~\eqref{Eq:CorWMd9} can be viewed as a
nonlinear wave equation with polynomial nonlinearity. Indeed, the
boundary condition~\eqref{Eq:BoundaryCondition} allows for a change of
variable $u(t,r)=rv(t,r)$ which leads to an eleven-dimensional radial
wave equation in $v$,
\begin{equation} \label{Eq:11_Dim_Wave}
\left
  (\partial_t^2-\partial_r^2-\frac{10}{r}\partial_r \right)
v(t,r)=-8\left [14v(t,r)^3-111r^2v(t,r)^5 \right].
\end{equation} In
fact, this is the point of view we adopt here. In particular, the
nonlinear term in Eq.~\eqref{Eq:CorWMd9} becomes smooth and
therefore admits a uniform Lipschitz estimate needed for a contraction
mapping argument, see Lemma \ref{Lem:Nonlinearity}. We also remark
that Eq.~\eqref{Eq:11_Dim_Wave}, in spite of its defocusing character
(at least for small values of $v$), admits an explicit self-similar
blowup solution. This is in stark contrast to the cubic defocusing
wave equation
\[\left (\partial_t^2-\partial_r^2-\frac{10}{r}\partial_r \right)v(t,r) =-v(t,r)^3 \]
for which no self-similar solutions exist. The self-similar blowup
in Eq.~\eqref{Eq:11_Dim_Wave} can therefore be
understood as a consequence of the presence of the focusing quintic
term which dominates the dynamics for large initial data.
 
\subsection{Main result} We start by intuitively describing the main
result. We fix $T_0>0$ and prescribe initial data $u[0]$ that
are close to $u^{T_0}[0]$ on a ball of radius slightly larger than
$T_0$. Here and throughout the paper we use the abbreviation
$ u[t]:= (u(t,\cdot), \partial_t u (t,\cdot))$. Then we prove the existence
of a particular $T$ near $T_0$ for which the solution $u$ converges to
$u^T$ inside the backward lightcone $\mc C_T$ in a norm adapted to the
blowup behavior of $u^T$.  For the precise statement of the main
result we use Definitions \ref{Def:Solution} and \ref{Def:BlowupTime}.
\begin{theorem}\label{Th:Main} Fix $T_0 > 0$. There exist constants
  $M,\delta,\epsilon > 0$ such that for any radial initial data $u[0]$
  satisfying
  \begin{align}\label{Eq:CondData}
    \bigg\||\cdot|^{-1}\bigg( u[0](|\cdot|)-  u^{T_0}[0](|\cdot|)\bigg) \bigg\|_{H^{6}(\mathbb{B}^{11}_{T_0+\delta}) \times H^{5}(\mathbb{B}^{11}_{T_0+\delta})}  \leq \frac{\delta}{M}
  \end{align}
  the following statements hold:
  \begin{enumerate}[i)]
  \item The blowup time at the origin $T :=T_{u[0]}$ belongs to the
    interval $[T_0 - \delta, T_0 + \delta]$.
  \item The solution $u: \mathcal{C}_T \to \R$ to
    Eq.~\eqref{Eq:CorWMd9} satisfies
    \begin{align}
      (T-t)^{-\frac{9}{2}+k} \bigg\||\cdot|^{-1}\bigg( u(t,|\cdot|) - u^T(t,|\cdot|)\bigg) \bigg\|_{\dot H^k (\mathbb{B}^{11}_{T-t})} &\leq \delta(T-t)^{\epsilon}, \label{Eq:Decay1}\\
      (T-t)^{-\frac{7}{2}+l} \bigg\||\cdot|^{-1}\bigg( \partial_t u(t,|\cdot|) - \partial_t u^T(t,|\cdot|)\bigg) \bigg\|_{\dot H^l (\mathbb{B}^{11}_{T-t})} &\leq \delta(T-t)^{\epsilon},\label{Eq:Decay2}
    \end{align}
    for integers $0\leq k\leq 6$ and $0\leq l \leq 5$. Furthermore,
    \begin{equation}\label{Eq:SupEstimate}
      \big\| u(t,\cdot) - u^T(t,\cdot)\big\|_{ L^{\infty}(0,T-t)} \leq \delta(T-t)^{\epsilon}.
    \end{equation}
  \end{enumerate}
	
\end{theorem}
\begin{remark}
  The normalizing factor on the left-hand side of~\eqref{Eq:Decay1}
  and~\eqref{Eq:Decay2} appears naturally as it reflects the behavior
  of the self-similar solution $u^T$ in the respective Sobolev norm,
  i.e.,
  \begin{align*}
    \big\|\,|\cdot|^{-1}u^T(t,|\cdot|) \big\|_{\dot H^k
      (\mathbb{B}^{11}_{T-t})}
&=\bigg\||\cdot|^{-1}\phi_0\left(\frac{|\cdot|}{T-t}\right)
\bigg\|_{\dot H^k (\mathbb{B}^{11}_{T-t})} \\
&=(T-t)^{\frac{9}{2}-k}\big\||\cdot|^{-1}\phi_0(|\cdot|) \big\|_{\dot H^k (\mathbb{B}^{11}_{1})}
  \end{align*} 
  and
  \begin{align*}
    \big\||\cdot|^{-1}\partial_tu_T(t,|\cdot|) \big\|_{\dot H^l
    (\mathbb{B}^{11}_{T-t})}
    &=(T-t)^{-2}\bigg\|\phi'_0\left(\frac{|\cdot|}{T-t}\right)
      \bigg\|_{\dot H^l (\mathbb{B}^{11}_{T-t})} \\
    &=(T-t)^{\frac{7}{2}-l}\big\|\phi'_0(|\cdot|) \big\|_{\dot H^l (\mathbb{B}^{11}_{1})}.
  \end{align*}
\end{remark}
\begin{remark}
  Since $\phi_0$ is monotonically increasing on $[0,1]$, we have
  \begin{equation}\label{Eq:Max_u}
    \big\| u^T(t,\cdot)\big\|_{ L^{\infty}(0,T-t)}=\max_{\rho\in[0,1]}\|\phi_0(\rho)\|=\phi_0(1).
  \end{equation} 
  Therefore, given $\ve_1>0$, it follows from~\eqref{Eq:SupEstimate}
  and~\eqref{Eq:Max_u} that $\delta$ can be chosen small enough so
  that
  \begin{align*}
    \big\| u(t,\cdot)\big\|_{ L^{\infty}(0,T-t)}
    &\leq \big\| u(t,\cdot) - u^T(t,\cdot)\big\|_{ L^{\infty}(0,T-t)}
      +\big\| u^T(t,\cdot)\big\|_{ L^{\infty}(0,T-t)}\leq\ve_1+\phi_0(1). 
  \end{align*}
  Hence, for $t< T$ the solution $u(t,r)$ stays inside a neighborhood
  of $u=0$ where the metric is given by~\eqref{Def:g}, i.e., the portion of
  the target manifold that participates in the dynamics of the blowup
  solution is described by the metric~\eqref{Def:g}.
\end{remark}

\subsection{Outline of the proof}
We use the method developed in the series of papers
\cite{Donninger11,DonnScho12,DonnScho14,Donninger14,DonnScho16,DonSch16,CDG17,Donninger17}. First,
we introduce the rescaled variables
\begin{equation}\label{Def:RescalVar1}
  v_1(t,r):=\frac{T-t}{r}u(t,r),\quad v_2(t,r):=\frac{(T-t)^2}{r} \partial_t u(t,r).
\end{equation}
Division by $r$ is justified by the boundary condition
\eqref{Eq:BoundaryCondition} and the presence
of the prefactors involving $T-t$ has to do with the change of
variables we subsequently introduce. That is, we introduce similarity
coordinates $(\tau,\rho)$ defined by 
\begin{align}\label{Def:SimCoord}
  \tau := -\log(T-t) + \log T, \quad \rho := \frac{r}{T-t},
\end{align}
and set
\begin{align}\label{Def:RescalVar2}
  \psi_j(\tau,\rho):=v_j(T(1-e^{-\tau}),T e^{-\tau}\rho)
\end{align}
for $j=1,2$. As a consequence, Eq.~\eqref{Eq:11_Dim_Wave} can be written as an abstract
evolution equation,
\begin{equation}\label{Eq:Abstract}
  \partial_{\tau}\Psi(\tau)=\mb L_0\Psi(\tau)+{\mb M}(\Psi(\tau)),
\end{equation}
where $\Psi(\tau)=(\psi_1(\tau,\cdot),\psi_2(\tau,\cdot))$, $\mb L_0$ is the spatial part of the
radial wave operator in the new coordinates, and $\mb M(\Psi(\tau))$ consists of
the remaining nonlinear terms.  The benefit of passing to the new variables
~\eqref{Def:SimCoord} and~\eqref{Def:RescalVar2} is that the backward
lightcone $\mc C_T$ is transformed into a cylinder
\[\mathcal{C}:= \{(\tau,\rho):\tau\in[0,\infty),\rho\in[0,1]\}, \]
the rescaled self-similar blowup solution $u^T$ becomes a
$\tau$-independent 
function $\Psi_{\text{res}}$ (this justifies the presence of
$t$-dependent prefactors in~\eqref{Def:RescalVar1}), and the problem of
stability of blowup transforms into the problem of 
asymptotic stability of a static solution.  We subsequently follow the
standard approach for studying the stability of steady-state solutions and
plug the ansatz $\Psi(\tau)=\Psi_{\text{res}}+\Phi(\tau)$ into
Eq.~\eqref{Eq:Abstract}. This leads to an evolution equation in $\Phi$,
\begin{equation}\label{Eq:AbsEvolPhi}
  \partial_{\tau}\Phi(\tau)=\mb L_0\Phi(\tau)+\mb{L'}\Phi(\tau)+{\mb N}(\Phi(\tau)),
\end{equation}
where $\mb L'$ is the Fr\'echet derivative of $\mb M$ at
$\Psi_{\text{res}}$ and $\mb N(\Phi(\tau))$ is the nonlinear remainder.  We
then proceed by studying Eq.~\eqref{Eq:AbsEvolPhi} as an ordinary
differential equation in a Hilbert space with the norm
\begin{align}\label{Def:Norm}
  \| \mb u \|^{2}  =  \left \| (u_1,u_2)  \right \|^{2}  
:= \| u_{1}(|\cdot|) \|_{H^{6} (\mathbb{B}^{11} )}^{2} + \| u_2(|\cdot|) \|_{H^{5} (\mathbb{B}^{11} ) }^{2}.
\end{align}

However, passing to new variables also comes with a
price. Namely, the radial wave operator $\mb L_0$ is
not self-adjoint. Nonetheless, we establish well-posedness of the
linearized problem (that is, Eq.~\eqref{Eq:AbsEvolPhi} with $\mb N$
removed) by using methods from semigroup
theory. In particular, we use an equivalent norm to~\eqref{Def:Norm}
and the Lumer-Phillips theorem to show that $\mb L_0$ generates a
semigroup $(\mb S_0(\tau))_{\tau \geq 0}$ with a negative growth
bound. This in particular allows for locating the spectrum of
$\mb L_0$. Furthermore, $\mb L'$ is compact so $\mb L:=\mb L_0+\mb L'$
generates a strongly continuous semigroup
$(\mb S(\tau))_{\tau \geq 0}$ and well-posedness of the
linearized problem follows.

The stability of the solution $u^T$ follows from a decay
estimate on the semigroup $\mb S(\tau)$. To obtain such an estimate we
exploit the relation between the growth bound of a semigroup and the
location of the spectrum of its generator. We therefore study
$\sigma(\mb L)$ which, thanks to the compactness of $\mb L'$, amounts to
studying the eigenvalue problem $(\la-\mb L)\mb u=0$. We subsequently
show that $\sigma(\mb L)$ is contained in the left half-plane except
for the point $\la=1$. However, this unstable eigenvalue corresponds
to an apparent instability and we later use it to fix the blowup
time. We therefore proceed by defining a spectral projection $\mb P$
onto the unstable space and study the semigroup $\mb S(\tau)$
restricted to $\rg(1-\mb P)$. Furthermore, we establish a uniform
bound on the resolvent $\mb R_{\mb L}(\la)$ and invoke the
Gearhart-Pr\"uss theorem to obtain a negative growth bound on
$(1-\mb P)\mb S(\tau)$.

Appealing to Duhamel's principle,
we rewrite Eq.~\eqref{Eq:AbsEvolPhi} in the integral form
\begin{equation}\label{Eq:Duhamel1}
  \Phi(\tau)=\mathbf{S}(\tau)\mathbf{U}(\mathbf{v},T)+\int_{0}^{\tau}\mathbf{S}(\tau-s)\mathbf{N}(\Phi(s))ds,
\end{equation}
where $\mb U(\mb v,T)$ represents the rescaled initial data. We remark
that the parameter $T$ does not appear in the equation itself but in
the initial data only. To obtain a decaying solution to
Eq.~\eqref{Eq:Duhamel1} we suppress the unstable part of $\mb S(\tau)$
by introducing a correction term
\[ \mb C(\Phi,\mb U(\mb v,T)):=\mb P\left(\mb U(\mb
  v,T)+\int_{0}^{\infty}e^{-s}\mb N(\Phi(s))ds \right) \]
into Eq.~\eqref{Eq:Duhamel1}. That is, we consider the
modified equation
\begin{align}\label{Eq:Modified1}
  \Phi(\tau)=\mb S(\tau)\big(\mb U(\mb v,T)-\mb C(\Phi,\mb U(\mb v,T))\big)  + \int_0^{\tau} \mb S(\tau - s)  \mb N(\Phi(s)) ds.
\end{align}
We subsequently prove that for a fixed $T_0$ and small enough initial
data $\mb v$, every $T$ close to $T_0$ yields a unique solution to
Eq.~\eqref{Eq:Modified1} that decays to zero at the linear decay
rate. In other words, we prove the existence of a solution curve to
Eq.~\eqref{Eq:Modified1} parametrized by $T$ inside a small
neighborhood of $T_0$, provided $\mb v$ is small enough.

Finally, we use the very presence of the unstable eigenvalue
$\la=1$ to prove the existence of a particular $T$ near $T_0$ for
which $\mb C(\Phi,\mb U(\mb v,T))=0$ and hence obtain a decaying
solution to Eq.~\eqref{Eq:Duhamel1} which, when translated back to the original
coordinates, implies the main result.

\subsection{Notation}
We denote by $\B^{d}_R$ the $d-$dimensional open ball of radius $R$
centered at the origin. For brevity we let $\B^d:=\B^d_1$. We write
2-component vector quantities in boldface, e.g. $\mb u=(u_1,u_2)$. By
$\mathcal{B}(\mathcal{H})$ we denote the space of bounded operators on
the Hilbert space $\mathcal{H}$. We denote by $\sigma(\mb L)$ and
$\sigma_p(\mb L)$ the spectrum and the point spectrum, respectively,
of a linear operator $\mb L$. Also, we denote by $\rho(\mb L)$ the
resolvent set $\C\setminus\sigma(\mb L)$ and use the 
convention $\mb R_{\mb L}(\la):=(\la-\mb L)^{-1}, \la\in\rho(\mb L)$,
for the resolvent operator. 
We use the
symbol $\lesssim$ with the standard meaning: $a\lesssim b$ if there
exists a positive constant $c$, independent of $a,b$, such that
$a\leq cb$. Also, $a\simeq b$ means that both $a\lesssim b$ and
$b\lesssim a$ hold.

\subsection{Similarity coordinates and cylinder formulation}
After introducing the similarity coordinates
\[ \tau := -\log(T-t) + \log T, \quad \rho := \frac{r}{T-t}, \]
and the rescaled variables
\begin{gather*}
  v_1(t,r):=\frac{T-t}{r}u(t,r),\quad v_2(t,r):=\frac{(T-t)^2}{r} \partial_t u(t,r),\\
  \psi_j(\tau,\rho)=v_j(T(1-e^{-\tau}),T\rho
  e^{-\tau}),\quad j=1,2,
\end{gather*}
 we obtain from Eq.~\eqref{Eq:CorWMd9} the first-order system
\begin{align}\label{Eq:MatrixEq}
  \begin{bmatrix} \partial_{\tau}\psi_1\vspace{1mm} \\ \partial_{\tau}\psi_2 \end{bmatrix}
  =\begin{bmatrix} -\rho\partial_\rho\psi_1-\psi_1+\psi_2 \vspace{1mm}\\ \partial^2_{\rho}\psi_1+\frac{10}{\rho}\partial_{\rho}\psi_1-\rho\partial_{\rho}\psi_2-2\psi_2 \end{bmatrix}
  -\begin{bmatrix} 0 \vspace{1mm}\\8( 14\psi_1^3-111\rho^2\psi_1^5)\end{bmatrix}
\end{align}
for $(\tau,\rho)\in\mathcal{C}$. Furthermore, the initial data become
\begin{align}\label{Eq:InitialData}
  \begin{bmatrix} \psi_1(0,\rho) \\ \psi_2(0,\rho) \end{bmatrix} 
  =\frac{1}{\rho}\begin{bmatrix} u_0(T\rho) \\ Tu_1(T\rho) \end{bmatrix}=\frac{1}{\rho}\begin{bmatrix}
    u^{T_0}(0,T\rho)\\ T\partial_{0}u^{T_0}(0,T\rho)\end{bmatrix}+\frac{1}{\rho}\begin{bmatrix}
    F(T\rho)\\TG(T\rho)
  \end{bmatrix}
\end{align}
where $T_0$ is a fixed parameter
and
\begin{align*}
  &F:=u_0-u^{T_0}(0,\cdot), \quad G:=u_1-\partial_{0}u^{T_0}(0,\cdot).
\end{align*}
In addition, we have the regularity conditions
\[ \partial_{\rho}\psi_1(\tau,\rho)|_{\rho=0}=\partial_{\rho}\psi_2(\tau,\rho)|_{\rho=0}=0 \]
for $\tau \geq 0$. 
Note further that we are studying the dynamics around $u^{T_0}$ for a fixed
$T_0$ and thus, it is natural to split the initial data as in
Eq.~\eqref{Eq:InitialData}. The parameter $T$ is assumed to be
close to $T_0$ and will be fixed later. As a consequence,
the proximity of the initial data to $u^{T_0}[0]$ is measured by
$\mb v:=(F,G).$

\subsection{Perturbations of the blowup
  solution}\label{Sec:PerturbedProblem} For convenience, we set
\[ \Psi(\tau)(\rho):=\begin{bmatrix}\psi_1(\tau,\rho)\\
  \psi_2(\tau,\rho)\end{bmatrix}.\]
In the rescaled variables the blowup solution $u^T$ becomes
$\tau$-independent, i.e., 
\begin{equation*}
  \begin{bmatrix}\frac{T-t}{r}u^T(t,r) \vspace{1mm} \\ \frac{(T-t)^2}{r} \partial_t u^T(t,r)\end{bmatrix}=
  \begin{bmatrix}\frac{1}{\rho}\phi_0(\rho) \vspace{1mm}\\\phi'_0(\rho)\end{bmatrix}=:\Psi_{\text{res}}(\tau)(\rho).
\end{equation*}
We proceed by studying the dynamics of Eq.~\eqref{Eq:MatrixEq} around
$\Psi_{\text{res}}$. Our aim is to prove the asymptotic stability of
$\Psi_{\text{res}}$ which in turn translates into the appropriate
notion of stability of $u^T$. We therefore follow the standard method
and plug the ansatz $\Psi=\Psi_{\text{res}}+\Phi$ into
Eq.~\eqref{Eq:MatrixEq}, where
$\Phi(\tau)(\rho):=(\vp_1(\tau,\rho),\vp_2(\tau,\rho))$. This leads to
an evolution equation for the perturbation $\Phi$,
\begin{equation}\label{Eq:PerturbedProblem}
  \begin{cases}
    \partial_{\tau}\Phi(\tau)=\widetilde{\bf{L}}\Phi(\tau)+{\bf{N}}(\Phi(\tau)),\\
    \Phi(0)={\bf{U}}(\textbf{v},T),
  \end{cases}
\end{equation}
where $\widetilde{\mb L}$ and $\mb N$ are spatial operators and
$\mb U(\mb v, T)$ are the initial data. More precisely,
$ \widetilde{\bf{L}}:=\widetilde{\bf{L}}_0+{\bf{L}'},$ where
\begin{gather}
  \widetilde{\bf{L}}_0\mb u(\rho):=
  \begin{bmatrix}\label{Eq:L0}
    -\rho u_1'(\rho)-u_1(\rho)+u_2(\rho) \vspace{1.5mm}\\
    u_1''(\rho)+\frac{10}{\rho}u_1'(\rho)-\rho u_2'(\rho)-2u_2(\rho)
  \end{bmatrix},\\
  {\bf{L}'}\mb u(\rho):=
  \begin{bmatrix}
    0\\
    W(\rho,\phi_0(\rho))u_1(\rho)
  \end{bmatrix}\label{Def:L'}
\end{gather}
and
\begin{align}
  {\bf{N}}(\mb u)(\rho):=
  \begin{bmatrix}
    0\\
    N(\rho,u_1(\rho))
  \end{bmatrix}\label{Eq:Nonlinearity}
\end{align}
for a 2-component function $\mb u(\rho)=(u_1(\rho),u_2(\rho))$, where
\begin{align}\label{Eq:WandN}
  \begin{split}
    &N(\rho,u_1(\rho))=-\frac{8}{\rho^3}[n(\phi_0(\rho)+\rho
    u_1(\rho))-n(\phi_0(\rho))-n'(\phi_0(\rho))\rho u_1(\rho)]\quad
    \text{and}\\
    &W(\rho,\phi_0(\rho))=-\frac{8}{\rho^2}n'(\phi_0(\rho))\quad\text{for}
    \quad n(x)=14x^3-111x^5.
  \end{split}
\end{align}
Also, we write the initial data as
\begin{equation}\label{Eq:InitialDataU}
  \Phi(0)(\rho)=\mb U(\mb v,T)(\rho)=\begin{bmatrix}\frac{1}{\rho}\phi_0\left(\frac{T}{T_0}\rho\right) \vspace{1mm}\\\frac{T^2}{T^2_0}\phi'_0\left(\frac{T}{T_0}\rho\right)\end{bmatrix}-\begin{bmatrix}\frac{1}{\rho}\phi_0(\rho) \vspace{1mm}\\\phi'_0(\rho)\end{bmatrix}+\mb V(\mb v,T)(\rho)
\end{equation}
where
\begin{equation*}
  \mb V(\mb v,T)(\rho):=\begin{bmatrix}\frac{1}{\rho}F(T\rho) \vspace{1mm}\\\frac{T}{\rho}G(T\rho)\end{bmatrix}, \quad \mb v=\begin{bmatrix}F \\ G\end{bmatrix}.
\end{equation*}

\subsection{Strong lightcone solutions and blowup time at the
  origin}

To proceed, we need the notion of a solution to the
problem~\eqref{Eq:PerturbedProblem}. In Section \ref{Sec:FunctSetting}
we introduce the space
\[ \mathcal{H}:= H^6_{\text{rad}}(\mathbb{B}^{11}) \times
H^5_{\text{rad}}(\mathbb{B}^{11})\]
and prove that the closure of the operator $\widetilde{\mb L}$, defined
on a suitable domain, generates a strongly continuous semigroup
$\mb S(\tau)$ on $\mathcal{H}$. Consequently, we formulate the
problem~\eqref{Eq:PerturbedProblem} as an abstract integral equation
via Duhamel's formula,
\begin{equation}\label{Eq:Duhamel}
  \Phi(\tau)=\mathbf{S}(\tau)\mathbf{U}(\mathbf{v},T)+\int_{0}^{\tau}\mathbf{S}(\tau-s)\mathbf{N}(\Phi(s))ds.
\end{equation}
This in particular establishes the well-posedness of the
problem~\eqref{Eq:PerturbedProblem} in $\mathcal{H}$. We are now in
the position to introduce the following definitions.
\begin{definition}\label{Def:Solution}
  We say that $u :\mathcal{C}_T\rightarrow\mathbb{R}$ is a
  \emph{solution} to ~\eqref{Eq:CorWMd9} if the corresponding
  $\Phi:[0,\infty)\rightarrow\mathcal{H}$ belongs to
  $C([0,\infty);\mathcal{H})$ and satisfies~\eqref{Eq:Duhamel} for all
  $\tau\geq0$.
\end{definition}
\begin{definition}\label{Def:BlowupTime}
  For the radial initial data $(u_0,u_1)$ we define
  $\mathcal{T}(u_0,u_1)$ as the set of all $T>0$ such that there
  exists a solution $u:\mathcal{C}_T\rightarrow\R$
  to~\eqref{Eq:CorWMd9}. We call
  \begin{equation}
    T_{(u_0,u_1)}:=\sup \left (\mathcal{T}(u_0,u_1)\cup\{0\}\right )
  \end{equation}
  the \emph{blowup time at the origin}.
\end{definition}

\subsection{Functional setting}\label{Sec:FunctSetting}

We consider radial Sobolev functions
$\hat{u}:\B^{11}_R\rightarrow\mathbb{C}$, i.e., $\hat{u}(\xi)=u(|\xi|)$
for $\xi\in\B^{11}_R$ and some $u: [0,R) \to \C$. We furthermore
define
$$u\in H_{\text{rad}}^m(\B^{11}_R) \enskip \text{if and only if} \enskip \hat{u}\in H^m(\B^{11}_R):=W^{m,2}(\B^{11}_R).$$
With the norm
\[
\|u\|_{H_{\text{rad}}^m(\B^{11}_R)}:=\|\hat{u}\|_{H^m(\B^{11}_R)}, \]
$H_{\text{rad}}^m(\B^{11}_R)$ becomes a Banach space.  In the rest of
this paper we do not distinguish between $u$ and $\hat{u}$. Now we
define the Hilbert space
\[ \mathcal{H}:= H^6_{\text{rad}}(\mathbb{B}^{11}) \times
H^5_{\text{rad}}(\mathbb{B}^{11}),\] with the induced norm
\begin{align*}
  \| \mathbf{u} \|^{2}  =  \left \| (u_{1},u_{2})  \right \|^{2}  := \| u_{1} \|_{H_{\text{rad}}^{6} (\mathbb{B}^{11} )}^{2} + \| u_{2} \|_{H_{\text{rad}}^{5} (\mathbb{B}^{11} ) }^{2}.
\end{align*}

\subsection{Well-posedness of the linearized equation}\label{Sec:Semigroup}
To establish well-posedness of the problem~\eqref{Eq:PerturbedProblem}
we start by defining the domain of the free operator
$\widetilde{\textbf{L}}_0$, see Eq.~\eqref{Eq:L0}. We follow
\cite{DonSch16} and let
\[ \mathcal{D}(\widetilde{\textbf{L}}_0):=\{ \mb u\in
C^{\infty}(0,1)^2\cap \mathcal{H}:w_1\in C^3[0,1],
w''_1(0)=0, w_2\in C^2 [0,1]\}\]
where
\begin{equation*}
  w_j(\rho):=D_{11}u_j(\rho):=\left(\frac{1}{\rho}\frac{d}{d\rho}\right)^{4}(\rho^9u_j(\rho))=\sum_{n=0}^{4}c_n
  \rho^{n+1}u^{(n)}_j(\rho).\end{equation*}
for certain positive constants $c_n$, $\rho\in[0,1]$, and $j=1,2$. Since $C^{\infty}(\overline{\B^{11}})$ is dense in $H^m(\B^{11})$,
\begin{align*}
  C_{\text{even}}^{\infty} [0,1]^2 := \Big \{  \mathbf{u} \in C^{\infty} [0,1]^2:~~\mathbf{u}^{(2k+1)}(0)=0,~~k=0,1,2,\dots \Big \}  \subset \mathcal{D} ( \widetilde{ \mathbf{L} }_{0})
\end{align*}
is dense in $\mathcal{H}$, which in turn implies that
$\widetilde{\textbf{L}}_0$ is densely defined on $\mathcal{H}.$
Furthermore, we have the following result.
\begin{proposition}\label{S0estimate}
  The operator
  $\widetilde{\bf{L}}_0:\mathcal{D}(\widetilde{\bf{L}}_0)\subset\mathcal{H}\rightarrow\mathcal{H}$
  is closable and its closure
  ${\bf{L}}_0:\mathcal{D}({\bf{L}}_0)\subset\mathcal{H}\rightarrow\mathcal{H}$
  generates a strongly continuous one-parameter semigroup
  $({\bf{S}}_0(\tau))_{\tau\geq0}$ of bounded operators on
  $\mathcal{H}$ satisfying the growth estimate
  \begin{equation}\label{Eq:S0Estimate}
    \|{\bf{S}}_0(\tau)\|\leq Me^{-\tau}
  \end{equation}
for all $\tau\geq 0$ and some $M>0$.
  Furthermore, the operator
  ${\bf{L}}:={\bf{L}}_0+{\bf{L}}':\mathcal{D}({\bf{L}})\subset\mathcal{H}\rightarrow\mathcal{H}$,
  $\mathcal{D}({\bf{L}})=\mathcal{D}({\bf{L}}_0)$, is the generator of
  a strongly continuous semigroup $({\bf{S}}(\tau))_{\tau\geq0}$ on
  $\mc H$ and $\mb L':\mc H\rightarrow \mc H$ is compact.
\end{proposition}
\begin{proof}
  The proof essentially follows the one of Proposition 3.1 in
  \cite{ChaDonnGlo17} for $d=9$.
\end{proof}
\subsection{The spectrum of the free operator}
By exploiting the relation between the growth bound of a semigroup and
the spectral bound of its generator, we can locate the spectrum of the
operator $\mb L_0$. Namely, according to \cite{EngelNagel00}, p. 55,
Theorem 1.10, the estimate~\eqref{Eq:S0Estimate} implies
\begin{equation}\label{Eq:FreeSpectrum}
  \sigma({\mb L_0})\subseteq\{ \lambda\in\mathbb{C}: \Re \lambda \leq-1 \} .
\end{equation}

\subsection{The spectrum of the full linear operator}
To understand the properties of the semigroup $\mb S(\tau)$ we
investigate the spectrum of the full linear operator $\mb L$. First of
all, we remark that $\la=1$ is an eigenvalue of $\mb L$ (see
Sec. \ref{Sec:g}), which is an artifact of the freedom of choice of
the parameter $T$, see e.g. \cite{CDG17} for a discussion on
this. What is more, $\la=1$ is the only spectral point of $\mb L$ with
a non-negative real part. To prove this we first focus on the point
spectrum.
\begin{proposition}\label{Prop:SpectrumL}
  We have
  \begin{equation}
    \sigma_p({\bf{L}})\subseteq\{ \lambda\in\mathbb{C}: \Re \lambda <0 \} \cup \{1\}.
  \end{equation}
\end{proposition}
\begin{proof} We argue by contradiction and assume there exists a
  $\la\in\sigma_p(\mb L)\setminus \{1\}$ with $\Re\la\geq0.$ This
  means that there exists a
  $\mb u=(u_1,u_2)\in \mathcal{D}(\mb L)\setminus\{0\}$ such that
  $\mb u\in \text{ker}(\la-\mb L).$ The spectral equation
  $(\la-\mb L)\mb u=0$ implies that the first component $u_1$
  satisfies the equation
  \begin{equation}\label{Eq:Eigenv}
    (1-\rho^2)u_1''(\rho)+\left(\frac{10}{\rho}-2(\la+2)\rho \right)u_1'(\rho)-(\la+1)(\la+2)u_1(\rho)-V(\rho)u_1(\rho)=0
  \end{equation}
  for $\rho\in(0,1)$, where
  \[
  V(\rho):=-W(\rho,\phi_0(\rho))=\frac{8n'(\phi_0(\rho))}{\rho^2}=-\frac{54(3737\rho^2-4340)}{(155-74\rho^2)^2}. \]
  Since $\mb u \in \mathcal{H}$, $u_1$ must be an element of
  $H_{\text{rad}}^{6}(\mathbb{B}^{11}).$ From the smoothness of the
  coefficients in~\eqref{Eq:Eigenv} we have an a priori regularity
  $u_1\in C^{\infty}(0,1)$.  In fact, we claim that
  $u_1\in C^{\infty}[0,1]$.  To show this, we use the Frobenius
  method. Namely, both $\rho=0$ and $\rho=1$ are regular singularities
  of Eq.~\eqref{Eq:Eigenv} and Frobenius' theory gives a series form of
  solutions locally around singular points.

  The Frobenius indices at $\rho=0$ are $s_1=0$ and
  $s_2=-9$. Therefore, two independent solutions of
  Eq.~\eqref{Eq:Eigenv} have the form
  \begin{equation*}
    u_1^1(\rho)=\sum_{i=0}^{\infty}a_i\rho^i \quad \text{and} \quad u_1^2(\rho)=C\log(\rho)u_1^1(\rho)+\rho^{-9}\sum_{i=0}^{\infty}b_i\rho^i
  \end{equation*}
  for some constant $C\in\mathbb{C}$ and $a_0=b_0=1$. Since $u_1^1(\rho)$ is
  analytic at $\rho=0$ and $u_1^2(\rho)$ does not belong to
  $H_{\text{rad}}^{6}(\mathbb{B}^{11}),$ we conclude that $u_1$ is a
  multiple of $u_1^1$ and therefore, $u_1\in C^{\infty}[0,1)$.
  
  The Frobenius indices at $\rho=1$ are
  $s_1=0$ and $s_2=4-\la$, and we distinguish different cases.  If
  $4-\la \notin \mathbb{Z}$ then the two linearly independent
  solutions are
  \begin{equation*}
    u_1^1(\rho)=\sum_{i=0}^{\infty}a_i(1-\rho)^i \quad \text{and} \quad u_1^2(\rho)=(1-\rho)^{4-\la}\sum_{i=0}^{\infty}b_i(1-\rho)^i
  \end{equation*} 
  with $a_0=b_0=1$. Since $u_1^1(\rho)$ is analytic at $\rho=1$ and
  $u_1^2$ does not belong to $H_{\text{rad}}^{6}(\mathbb{B}^{11}),$ we
  conclude that $u_1\in C^{\infty}[0,1]$. If $4-\la \in \mathbb{N}_0$,
  then the fundamental solutions around $\rho=1$ are of the form
  \begin{equation*}
    u_1^1(\rho)=(1-\rho)^{4-\la}\sum_{i=0}^{\infty}a_i(1-\rho)^i \quad \text{and} \quad u_1^2(\rho)=\sum_{i=0}^{\infty}b_i(1-\rho)^i + C\log(1-\rho)u_1^1(\rho),
  \end{equation*}
  with $a_0=b_0=1$. Since $u_1^1(\rho)$ is analytic at $\rho=1$ and
  $u_1^2$ does not belong to $H_{\text{rad}}^{6}(\mathbb{B}^{11})$
  unless $C=0$, we again conclude that $u_1\in
  C^{\infty}[0,1]$.
  Finally, if $4-\la$ is a negative integer, the linearly independent
  solutions around $\rho=1$ are
  \begin{equation*}
    u_1^1(\rho)=\sum_{i=0}^{\infty}a_i(1-\rho)^i \quad \text{and} \quad u_1^2(\rho)=(1-\rho)^{4-\la}\sum_{i=0}^{\infty}b_i(1-\rho)^i+C\log(1-\rho)u_1^1(\rho),
  \end{equation*}
  with $a_0=b_0=1$. Once again, since $u_1^1(\rho)$ is analytic at
  $\rho=1$ and $u_1^2$ is not a member of
  $H_{\text{rad}}^{6}(\mathbb{B}^{11})$, we infer that
  $u_1\in C^{\infty}[0,1]$.

 To obtain the desired contradiction, it remains to prove
  that Eq.~\eqref{Eq:Eigenv} does not have a solution in
  $C^{\infty}[0,1]$ for $\Re\la\geq 0$ and $\la\neq1$. This claim goes
  under the name of the \emph{mode stability} of the solution $u^T$. A
  general approach to proving mode stability of explicit
  self-similar blowup solutions to nonlinear wave equations of the
  type~\eqref{Eq:CorWMgen} was developed in \cite{CDGH16,CDG17}. We
  argue here along the lines of \cite{CDG17}. Also, for the rest of
  the proof, we follow the terminology of \cite{CDG17}. Namely, we
  call $\la\in \C$ an \emph{eigenvalue} if it yields a
  $C^{\infty}[0,1]$ solution to the equation in question. Also, if an
  eigenvalue $\la$ satisfies $\Re\la\geq0$ we say it is
  \emph{unstable}, otherwise we call it \emph{stable}. Our aim is
  therefore to prove that, apart from $\la=1$, there are no unstable
  eigenvalues of the problem~\eqref{Eq:Eigenv}.

  First of all, we make the substitution
  $v(\rho)=\rho u_1(\rho)$. This leads to the equation
  \begin{equation}\label{eq:eigen1}
    (1-\rho^2)v''(\rho)+\left(\frac{8}{\rho}-2(\la+1)\rho \right)v'(\rho)-\la(\la+1)v(\rho)-\hat{V}(\rho)v(\rho)=0,
  \end{equation}
  where
  \[
  \hat{V}(\rho):=-\frac{10(15799\rho^4-5084\rho^2-19220)}{\rho^2(155-74\rho^2)^2}. \]
  Now we formulate the corresponding supersymmetric problem,
  \begin{equation}\label{eq:SUSY}
    (1-\rho^2)\tilde{v}''(\rho)+\left(\frac{8}{\rho}-2(\la+1)\rho \right)\tilde{v}'(\rho)-(\la+2)(\la-1)\tilde{v}(\rho)-\tilde{V}(\rho)\tilde{v}(\rho)=0,
  \end{equation}
  where
  \[
  \tilde{V}(\rho):=-\frac{18(3737\rho^4+5735\rho^2-24025)}{\rho^2(155-74\rho^2)^2}, \]
see~\cite{CDG17}, Sec.~3.2, for the derivation.
  We claim that, apart from $\la=1$, Eqs.~\eqref{eq:eigen1}
  and~\eqref{eq:SUSY} have the same set of unstable eigenvalues. This
  is proved by a straightforward adaptation of the proof of
 Proposition 3.1 in
  \cite{CDG17}.

   To establish the nonexistence of unstable eigenvalues of
  the supersymmetric problem~\eqref{eq:SUSY} we follow the proof of
  Theorem 4.1 in \cite{CDG17}.  We start by introducing the change of
  variables
  \begin{equation}\label{eq:change_of_var}
    x=\rho^2, \quad \tilde{v}(\rho)=\frac{x}{\sqrt{155-74x}}y(x).
  \end{equation}
Eq.~\eqref{eq:SUSY} transforms into Heun's equation in its
  canonical form,
  \begin{equation}\label{eq:heunform}
    y''(x)+\left(\frac{13}{2x}+\frac{\la-3}{x-1}-\frac{74}{74x-155}\right)y'(x)
+\frac{74\la(\la+3)x-(155\la^{2}+775\la+1656)}{4x(x-1)(74x-155)}y(x)=0.
  \end{equation} \normalsize
  Note that~\eqref{eq:change_of_var} preserves the analyticity of solutions at 0 and 1, and consequently, equations~\eqref{eq:SUSY} and~\eqref{eq:heunform} have the same set of eigenvalues.
  The Frobenius indices of Eq.~\eqref{eq:heunform} at $x=0$ are $s_1=0$ and $s_2=-\frac{11}{2}$, so its normalized analytic solution at $x=0$ is given by the power series
  \begin{equation}\label{power_series_at_0}
    \sum_{n=0}^{\infty}a_n(\lambda)x^{n},\quad a_0(\lambda)=1. 
  \end{equation}
  The strategy is to study the asymptotic behavior of the coefficients
  $a_n(\la)$ as $n\to\infty$. More precisely, we prove that if
  $\la\in\Hb$\footnote{Here, as in \cite{CDG17}, $\Hb$ denotes the
    closed complex right half-plane.} then
  $\lim_{n\rightarrow\infty}a_n(\la)=1$. Since $x=1$ is the only
  singular point of Eq.~\eqref{eq:heunform} on the unit circle, it
  follows that
  the solution given by the series~\eqref{power_series_at_0} is not
  analytic at $x=1$.

 First, we obtain the recurrence relation for coefficients
  $\{a_n(\la)\}_{n\in\mathbb{N}_0}$. By
  inserting~\eqref{power_series_at_0} into Eq.~\eqref{eq:heunform} we
  get
  \begin{align*}
    310&(2n+15)(n+2)\,a_{n+2}(\la)=\\  
       &[155\la(\la+4n+9)+2(458n^2+2357n+2727)]a_{n+1}(\la)-74(\la+2n+3)(\la+2n)a_{n}(\la),
  \end{align*}
  where $a_{-1}(\lambda)=0$ and $a_0(\lambda)=1$, or, written differently,
  \begin{equation}\label{eq:rec_a}
    a_{n+2}(\la)=A_n(\la)\,a_{n+1}(\la)+B_n(\la)\,a_n(\la),
  \end{equation}
  where
  \[
  A_n(\la)=\frac{155\la(\la+4n+9)+2(458n^2+2357n+2727)
  }{310(2n+15)(n+2)}
  \]
  and
  \[
  B_n(\la)=\frac{-37(\la+2n+3)(\la+2n)}{155(2n+15)(n+2)}.
  \]
  We now let
  \begin{equation}\label{eq:def_of_r}
    r_n(\la)=\frac{a_{n+1}(\la)}{a_n(\la)},
  \end{equation}
  and thereby transform Eq.~\eqref{eq:rec_a} into
  \begin{equation}\label{eq:rec_r}
    r_{n+1}(\la)=A_n(\la)+\frac{B_n(\la)}{r_n(\la)},
  \end{equation}
  with the initial condition
  \[r_0(\la)=\frac{a_1(\la)}{a_0(\la)}=A_{-1}(\la)=\frac{1}{26}\la^2+\frac{5}{26}\la+\frac{828}{2015}. \]
  Analogous to Lemma 4.2 in \cite{CDG17} we have that, given
  $\la\in\Hb$, either
  \begin{align}\label{Eq:Limit1}
    \lim_{n\rightarrow \infty} r_n(\la) = 1
  \end{align}
  or
  \begin{equation}\label{Eq:Limit2}
    \lim_{n\rightarrow \infty} r_n(\la) = \frac{74}{155}.
  \end{equation}
  Our aim is to prove that~\eqref{Eq:Limit1} holds throughout
  $\Hb$. We do that by approximately solving Eq.~\eqref{eq:rec_r} for
  $\la\in\Hb$. Namely, we define an approximate solution (also called
  quasi-solution) 
  \begin{equation*}\label{eq:quasisol}
    \tilde{r}_n(\lambda)=\frac{\lambda^2}{4n^2+28n+27}+\frac{\lambda}{n+7}+\frac{2n+12}{2n+23}
  \end{equation*}
to Eq.~\eqref{eq:rec_r},
  see \cite{CDGH16}, \S 4.1 for a discussion on how to obtain such an
  expression. Subsequently, we let
  \begin{equation}\label{Def:Delta}
    \delta_n(\la)=\frac{r_n(\la)}{\tilde{r}_n(\la)}-1
  \end{equation}
  and from Eq.~\eqref{eq:rec_r} we get the 
  recurrence relation 
  \begin{equation}\label{delta_recurrence}
    \delta_{n+1}=\ve_n-C_n\frac{\delta_n}{1+\delta_n}
  \end{equation}
for $\delta_n$,
  where
  \begin{equation}\label{epsilon_and_C}
    \ve_n=\frac{A_n\tilde{r}_n+B_n}{\tilde{r}_n\tilde{r}_{n+1}}-1 \quad \text{and} 
    \quad C_n=\frac{B_n}{\tilde{r}_n\tilde{r}_{n+1}}.
  \end{equation}
  Now, for all $\la\in\Hb$ and $n\geq 7$ we have the bounds
  \begin{align}\label{eq:estimates}
    |\delta_7(\la)|\leq\frac{1}{3}, \quad |\ve_n(\la)|\leq\frac{1}{12}, \quad |C_n(\la)|\leq\frac{1}{2}. 
  \end{align}
  The last two inequalities above are proved in the same way as the
  corresponding ones in Lemma 4.4 in \cite{CDG17}. However, the proof
  of the first one needs to be slightly adjusted and we provide it in
  the appendix, see Proposition \ref{Prop:AppendixEstimate}.  Next, by
  a simple inductive argument we conclude
  from~\eqref{delta_recurrence} and~\eqref{eq:estimates} that
  \begin{equation}\label{Eq:FinalEst}
    |\delta_n(\la)|\leq\frac{1}{3} \quad \text{for all }n\geq7 \text{ and }\la\in\Hb.
  \end{equation}
  Since for any fixed $\la\in\Hb$,
  $\lim_{n\rightarrow\infty}\tilde{r}_n(\la)=1$,
\eqref{Eq:FinalEst} and~\eqref{Def:Delta} exclude the
  case~\eqref{Eq:Limit2}. Hence,~\eqref{Eq:Limit1} holds
  throughout $\Hb$ and we conclude that there are no unstable
  eigenvalues of the supersymmetric problem~\eqref{eq:SUSY}, thus
  arriving at a contradiction and thereby completing the proof of the
  proposition.
\end{proof}

\begin{remark}
  Apart from $\la=1$ the point spectrum of the operator $\mb L$ is
  completely contained in the open left half plane. It is natural to
  try to locate the eigenvalues that are closest to the imaginary axis
  as their location is typically related to the rate of convergence to
  the blowup solution $u^T$. Our numerical calculations indicate that
  $-0.98\pm 3.76\,i$ is the approximate location of the pair of
  (complex conjugate) stable eigenvalues with the largest real
  parts. It is interesting to contrast this with the analogous spectral problems for
  equivariant wave maps into the sphere and Yang-Mills fields, where all
  eigenvalues appear to be real, see \cite{BizonBiernat15}.
\end{remark}

\begin{corollary}
  \label{Cor:Spectrum_L}
  We have
  \begin{align*}
    \sigma (\mathbf{L}) \subseteq \{ \lambda \in \mathbb{C}:~~\mathrm{Re} \lambda <0 \} \cup \{1\}.
  \end{align*}
\end{corollary}

\begin{proof}
  Assume there exists a $\lambda\in \sigma(\mathbf L)\setminus \{1\}$
  with $\mathrm{Re}\lambda\geq 0$. From~\eqref{Eq:FreeSpectrum} we see
  that $\la$ is contained in the resolvent set of $\mb L_0$.
  Therefore, we have the identity
  \begin{equation}\label{Eq:ResolventIdentity}
    \lambda-\mathbf L=[1-\mathbf L'\mathbf R_{\mathbf L_0}(\lambda)](\lambda-\mathbf L_0).
  \end{equation}
  This implies that $1\in \sigma (\mb L' \mb R_{\mb L_0}(\la))$ and
  since $\mathbf L'\mb R_{\mb L_0}(\lambda)$ is compact, it follows
  that $1\in \sigma_p(\mb L'\mb R_{\mathbf L_0}(\lambda))$. Thus,
  there exists a nontrivial $\mb f \in \mathcal H$ such that
  $[1-\mb L'\mb R_{\mb L_0}(\lambda)]\mb f=0$.  Consequently,
  $\mb u:=\mb R_{\mb L_0}(\lambda)\mb f\not= 0$ satisfies
  $(\lambda-\mb L)\mathbf u=0$ and thus, $\lambda\in \sigma_p(\mb L)$,
  but this is in conflict with Proposition \ref{Prop:SpectrumL}.
\end{proof}

\subsection{The eigenspace of the isolated eigenvalue}\label{Sec:g}
In this section, we prove that the (geometric) eigenspace of the
isolated eigenvalue $\lambda=1$ for the full linear operator
$\mathbf{L}$ is spanned by
\begin{equation}\label{Eq:g}
  \mb g(\rho):=\begin{bmatrix}g_1(\rho)\\g_2(\rho)\end{bmatrix}=\begin{bmatrix}\phi'_0(\rho)\\\rho\phi''_0(\rho)+2\phi'_0(\rho) \end{bmatrix}.
\end{equation}
Namely, we are looking for all
$\mathbf{u}=(u_{1},u_{2}) \in \mathcal{D} (\mathbf{L}) \setminus \{ 0
\}$
which belong to $\ker(1-\mathbf{L})$. A straightforward calculation
shows that the spectral equation $(1-\mathbf{L}) \mathbf{u} =0$ is
equivalent to the following system of ordinary differential equations,
\begin{align}\label{Eq:System}
  \begin{cases}
    u_2(\rho)=\rho u_1'(\rho)+2u_1(\rho),&\\
    \big( 1-\rho^2 \big) u_1''(\rho) + \Big( \frac{10}{\rho}-6\rho
    \Big) u_1'(\rho)-\Big( 6+\frac{8}{\rho^2} n'\big(\phi_0(\rho)\big)
    \Big ) u_1 (\rho) =0, \
  \end{cases}
\end{align} 
for $\rho \in (0,1)$. One can easily verify that a fundamental system of the
second equation is given by the functions $\phi'_0(\rho)$ and
$\rho^{-9}A(\rho)$, where $A(\rho)$ is analytic and non-vanishing at
$\rho=0$. We can therefore write the general solution to the second
equation as
\begin{align*}
  u_{1} (\rho) = C_{1}  \phi'_0(\rho) + C_{2} \frac{A(\rho)}{\rho ^9}.
\end{align*}
The condition $\mathbf{u} \in \mathcal{D}( \mathbf{L})$ requires $u_{1}$
to lie in the Sobolev space $H_{\text{rad}}^{6} (\mathbb{B}^{11}
)$.
Since $\phi'_0\in C^{\infty}[0,1]$, this requirement yields $C_{2}=0$
which, according to the first equation in~\eqref{Eq:System}, gives
$\mathbf u = C_1 \mathbf{g} $. In conclusion,
\begin{align} \label{ker} \text{ker}(1-\mathbf{L}) = \langle
  \mathbf{g} \rangle,
\end{align}
as initially claimed.

\subsection{Time evolution of the linearized problem}

To get around the spurious instability on the linear level, we use the
fact that $\la=1$ is isolated to introduce a (non-orthogonal) spectral
projection $\mb P$ and study the subspace semigroup $\mb S(\tau)(1-\mb
P)$. From Corollary \ref{Cor:Spectrum_L} we then
infer that the spectrum of its generator is contained in the
left-half plane. This does not necessarily imply the desired decay on
$\mb S(\tau)(1-\mb P)$. We nonetheless establish such a decay by first
proving uniform boundedness of the resolvent of $\mb L$ in a
half-plane that strictly contains $\Hb$ and then using the
Gearhart-Pr\"uss theorem. For this purpose, we define
\begin{align*}
  \Omega _{\ve,R} :=  \{  \lambda \in \mathbb{C}:~~~\text{Re}\lambda \geq -1+\ve, |\lambda|\geq R \}  
\end{align*}
for $\ve, R >0$.

\begin{proposition}\label{Prop:UnifBoundRes}
  Let $\ve >0$. Then there exists a constant $R_{\epsilon}>0$ such
  that the resolvent $\mathbf{R}_{\mathbf{L}}$ exists on
  $\Omega _{\ve,R_\ve}$ and satisfies
  \begin{align*}
    \| \mathbf{R} _{\mathbf{L} } (\lambda)\| \leq \frac{2}{\ve}
  \end{align*}
  for all $\lambda\in \Omega_{\ve,R_\ve}$.
\end{proposition}

\begin{proof}
  Fix $\ve>0$ and take $\lambda\in \Omega_{\ve,R}$ for an arbitrary
  $R\geq 2$.  Then $\lambda\in \rho(\mathbf L_0)$ and the
  identity~\eqref{Eq:ResolventIdentity} holds.  The proof proceeds as
  follows. For large enough $R$, we show that the operator
  $1-\mathbf L'\mathbf R_{\mathbf L_0}(\lambda)$ is invertible in
  $\Omega_{\ve,R}$ and $\mb R_{\mathbf L_0}(\la)$ and
  $[1-\mathbf L'\mathbf R_{\mathbf L_0}(\lambda)]^{-1}$ are uniformly
  norm bounded there. Via~\eqref{Eq:ResolventIdentity} this implies
  the desired bound on $\mb R_{\mathbf L}(\lambda)$.
	
   First of all, semigroup theory yields the estimate
  \begin{equation}\label{Eq:Res_Estimate}
    \|\mathbf R_{\mathbf L_0}(\lambda) \|\leq \frac{1}{\mathrm {Re}\lambda+1},
  \end{equation} 
  see \cite{EngelNagel00}, p.~55, Theorem 1.10.	
Next, by a Neumann series argument, the operator
  $1-\mathbf L'\mathbf R_{\mathbf L_0}(\lambda)$ is invertible if
  $\|\mathbf L'\mathbf R_{\mathbf L_0}(\lambda)\|<1$. To prove
  smallness of
  $\mathbf{L}^{\prime} \mathbf{R}_{\mathbf{L}_{0}} (\lambda)$, we
  recall the definition of $\mathbf L'$, Eq.~\eqref{Def:L'},
  \begin{align*}
    {\bf{L}'}\mb u(\rho):=
    \begin{bmatrix} 
      0\\
      \tilde{W}(\rho)u_1(\rho)
    \end{bmatrix}, ~~
    \tilde{W}(\rho)=-\frac{8}{\rho^2}n'(\phi_0(\rho))\enskip
    \text{for} \enskip n(x)=14x^3-111x^5.
  \end{align*}
  Let $\mathbf u=\mathbf R_{\mathbf L_0}(\lambda)\mathbf f$ or,
  equivalently, $(\lambda - \mathbf{L}_{0}) \mathbf{u}
  =\mathbf{f}$. The latter equation implies
  \begin{align*}
    (\lambda +1) u_{1} (\rho )  =u_{2} (\rho)- \rho u_{1}^{\prime} (\rho) + f_{1} (\rho).
  \end{align*}
  Now we use Lemma 4.1 from \cite{DonSch16} and
  $\|\tilde{W}^{(k)}\|_{L^\infty(0,1)}\lesssim 1$ for all
  $k\in \{0,1,\dots,5\}$ to obtain
  \begin{align*}
    |\la +1| \| \mb L^{\prime} \mb R_{\mb L_{0}} (\la) \mb f  \| & = |\la +1|  \| \mb{L}^{\prime} \mb{u} \|  \simeq \big \| \tilde{W} \big( u_{2} - (\cdot) u_{1}^{\prime} +f_{1} \big) \big \|_{ H_{\text{rad}}^{5} (\B^{11} ) } \\
                                                                 &  \lesssim \| u_{2} \|_{ H_{\text{rad}}^{5} (\B^{11} ) }  + \| (\cdot) u_{1}^{\prime} \|_{ H_{\text{rad}}^{5} (\B^{11} ) }  + \| f_{1} \|_{ H_{\text{rad}}^{5} (\mathbb{B}^{11} ) } \\
                                                                 &  \lesssim \| u_{2} \|_{ H_{\text{rad}}^{5} (\B^{11} ) }  + \|  u_{1} \|_{ H_{\text{rad}}^{6} (\B^{11} ) }  + \| f_{1} \|_{ H_{\text{rad}}^{6} (\B^{11} ) } \\
                                                                 & \lesssim  \| \mb u \| + \| \mb f \|  \lesssim \Big( \frac{1}{\text{Re}\la +1}  + 1 \Big) \| \mb f \| \\
                                                                 &\lesssim \|\mb f\|,
  \end{align*}
  where we used~\eqref{Eq:Res_Estimate}.  In other words,
  \begin{align*}
    \| \mathbf{L}^{\prime} \mathbf{R}_{\mathbf{L} _{0}} (\lambda)   \| \lesssim \frac{1}{ |\lambda +1| } \leq \frac{1}{|\lambda| -1} \leq \frac{1}{R-1}
  \end{align*}
  and by choosing $R$ sufficiently large, we can achieve
  $\|\mathbf L'\mathbf R_{\mathbf L_0}(\lambda)\|\leq\frac{1}{2}$.  As
  a consequence,
  $[1- \mathbf{L}^{\prime} \mathbf{R}_{\mathbf{L} _{0}} (\lambda)
  ]^{-1}$
  exists for $\la\in\Omega_{\ve,R_{\ve}}$ and we obtain the bound
  \begin{align*}
    \| \mb R_{\mb L} (\la) \| & = \| \mb R_{\mb L_{0}} (\la) [1-\mb L' \mb R_{\mb L_{0}} (\lambda) ]^{-1} \| \\
                              & \leq  \| \mb R_{\mb L_{0}} (\la) \| \| [1-  \mb L' \mb R_{\mb L_{0}} (\la)]^{-1} \| \\
                              & \leq  \| \mb R_{\mb L_{0}} (\la) \| \sum _{n=0}^{\infty} \| \mb L' \mb R_{\mb L_{0}} (\la) \|^n \leq \frac{2}{\ve}.
  \end{align*}
\end{proof}
We now show the existence of a projection $\mb P$ which
decomposes the Hilbert space $\mathcal{H}$ into a stable and an unstable
subspace and furthermore prove that data from the stable subspace lead
to solutions that decay exponentially in time. We also remark that it
is crucial to ensure that
$\rank \mb P = 1$, i.e., that $\mb g$ is the only unstable direction in
$\mathcal{H}.$
\begin{proposition}
  There exists a projection operator
  \begin{equation*}
    \mathbf{P}\in\mathcal{B}(\mathcal{H}), \enskip \mathbf{P}:\mathcal{H}\rightarrow \langle \mb g \rangle,
  \end{equation*}
  which commutes with the semigroup $({\bf{S}}(\tau))_{\tau\geq0}$. In
  addition, we have
  \begin{align}\label{Eq:SemigroupUnstable}
    \mathbf{S}(\tau)\mathbf{Pf}=e^{\tau}\mathbf{Pf}
  \end{align}
  and there are constants $C,\epsilon>0$ such that
  \begin{equation}\label{Eq:SemigroupDecay}
    \| (1-\mathbf{P})\mathbf{S}(\tau)\mathbf{f} \|\leq Ce^{-\epsilon\tau}\|(1-\mathbf{P})\mathbf{f} \|
  \end{equation}
  for all $\mathbf{f}\in \mathcal{H}$ and $\tau\geq0$.
\end{proposition}
\begin{proof}
  By Corollary \ref{Prop:SpectrumL}, the eigenvalue $\la=1$ of the
  operator $\mb L$ is isolated. We therefore introduce the
  spectral projection
  \[ \mb P: \mathcal{H}\rightarrow\mathcal{H}, \quad \mb P:=
  \frac{1}{2\pi i}\int_{\gamma}\mb R_{\mb L}(\mu)d\mu, \]
  where $\gamma$ is a positively oriented circle around $\la=1$. The
  radius of the circle is chosen small enough so that $\gamma$ is
  completely contained inside the resolvent set of $\mb L$ and such that
  the interior of $\gamma$ contains no spectral points of $\mb L$
  other than $\la=1$, see e.g.~\cite{Kato95}.  The projection $\mb P$ commutes
  with the operator $\mathbf{L}$ and therefore with the semigroup
  $\mb S(\tau)$. Moreover, the Hilbert space $\mathcal{H}$ is
  decomposed as $\mathcal{H} =\mathcal M \oplus \mathcal N$, where
  $\mathcal M:=\rg \mb P$ and $\mathcal N:=\rg(1-\mb P)=\ker \mb
  P$.
  Also, the spaces $\mathcal{M}$ and $\mathcal{N}$ reduce the operator
  $\mb L$ which is therefore decomposed into $\mb L_{\mathcal M}$ and
  $\mb L_{\mathcal N}$. The spectra of these operators are given by
  \begin{align} \label{Eq:SplitSpectrum} \sigma \left( \mb L_{\mathcal
        N} \right) = \sigma (\mb L) \setminus \{1\},\quad\quad \sigma
    \left( \mb L_{\mathcal M} \right ) = \{1\}.
  \end{align}
  We refer the reader to \cite{Kato95} for these standard results.
	
 To proceed with the proof we show that
  $\rank\mb P:=\dim\rg\mb P<+\infty$. We argue by contradiction and
  assume that $\rank\mb P=+\infty$. This means that $\la=1$ belongs to
  the essential spectrum of $\mb L$, see \cite{Kato95}, p.~239,
  Theorem 5.28. But according to Proposition $\ref{S0estimate}$ the
  operator $\mb L_0=\mb L-\mb L'$ is a compact perturbation of
  $\mb L$, and due to the stability of the essential spectrum under
  compact perturbations we conclude that $\la=1$ is a spectral point
  of $\mb L_0$. However, this is in conflict
  with~\eqref{Eq:FreeSpectrum}, and therefore $\rank\mb P<+\infty$.
	
 Now we prove that
  $\langle \mb g \rangle=\mathrm{rg}\,\mb P$. From the definition of
  the projection $\mb P$ we have $\mb P\mb g= \mb g$. Therefore
  $\mb \langle \mb g \rangle \subseteq \rg \mb P$ and it remains to
  prove the reverse inclusion. From the fact that the operator
  $1-\mb L_{\mathcal M}$ acts on the finite-dimensional Hilbert space
  $\mathcal M=\rg \mb P$ and $\eqref{Eq:SplitSpectrum}$ we infer that
  $\lambda =0$ is the only spectral point of $1-\mb L_{\mathcal
    M}$.
  Hence, $1-\mb L_{\mathcal M}$ is nilpotent, i.e., there exists a
  $k\in \mathbb{N}$ such that
  \begin{align*}
    \big( 1-\mb L_{\mathcal{M}} \big)^{k} \mb u= 0
  \end{align*}
  for all $\mb u \in \mathrm{rg}\, \mb P$ and we assume $k$ to be
  minimal.  Due to $\eqref{ker}$ the claim follows immediately for
  $k=1$. We therefore assume that $k\geq 2$. This implies the
  existence of a nontrivial function
  $\mb u \in \rg \mb P \subseteq \mathcal{D}( \mb L)$ such that
  $(1-\mb L_{\mathcal M})\mb u$ is nonzero and belongs to
  $\ker(1-\mb L_{\mathcal M})\subseteq \ker(1-\mb L)=\langle\mb
  g\rangle$.
  Therefore $(1-\mb L) \mb u = \alpha \mb g$, for some
  $\alpha \in \mathbb C \setminus \{0\}$. For convenience and without loss
  of generality we set $\alpha=-1$. By a straightforward computation
  we see that the first component of $\mb u$ satisfies the 
  differential equation
  \begin{align}\label{Eq:Nonhomog}
    \left(1-\rho ^2\right) u_{1}^{\prime \prime} (\rho) + \left( \frac{10}{\rho} -6 \rho \right ) u_{1} ^{\prime} (\rho) - \left ( 6 + \frac{8}{\rho^2} n'(\phi_0(\rho)) \right ) u_{1} (\rho) =  G(\rho),
  \end{align}
  for $\rho \in (0,1)$, where
  \begin{align*}
    G(\rho):= 2\rho\phi_0''(\rho)+5\phi_0'(\rho),\quad\rho \in [0,1].
  \end{align*}
  To find a general solution to Eq.~\eqref{Eq:Nonhomog} we first
  observe that
  \begin{align*}
    \hat{u}_{1} (\rho):= g_{1} (\rho) = \phi_0'(\rho),\quad \rho \in (0,1)
  \end{align*}
  is a particular solution to the homogeneous equation
  \begin{align*}
    \left(1-\rho ^2\right) u_1'' (\rho) + \left( \frac{10}{\rho} -6 \rho \right ) u_1' (\rho) - \left ( 6 + \frac{8}{\rho^2} n'\left(\phi_0(\rho)\right) \right ) u_{1} (\rho) =  0,
  \end{align*}
  see~\eqref{Eq:g} and~\eqref{Eq:System}.  Note that the Wronskian for
  the equation above is
  \begin{align*}
    \mathcal{W} (\rho) := \frac{(1-\rho ^2)^{2}}{\rho ^{10}}.
  \end{align*}
  Therefore, another linearly independent solution is
  \begin{align*}
    \hat{u}_2(\rho) := \hat{u}_1 (\rho) \int_{\rho}^{1}  \frac{(1-x^2)^2}{x^{10}} \frac{1}{\phi_0'(x)^2} dx,
  \end{align*}
  for all $\rho\in(0,1)$.  Note
  that near $\rho=0$ we have the expansion
  \[ \hat{u}_2(\rho)=\frac{1}{\rho^9}\sum_{j=0}^\infty a_j\rho^j,\quad
  a_0\neq 0, \]
  as already indicated in Section \ref{Sec:g}. Furthermore, we have
  \[ \hat{u}_2(\rho)=(1-\rho)^{3}\sum_{j=0}^\infty b_j
  (1-\rho)^j,\quad b_0\neq 0, \]
  near $\rho=1$.  Now, by the variation of constants formula we see
  that the general solution to Eq.~\eqref{Eq:Nonhomog} can be written
  as
  \begin{align*}
    u_{1} (\rho) = c_{1}   \hat{u}_{1} (\rho) + c_{2} \hat{u}_{2} (\rho)
    + \hat{u}_{2} (\rho) \int _{0}^{\rho} \frac{ \hat{u}_{1}(y)G(y)y^{10} }{(1-y^2)^{3}} dy  - \hat{u}_{1} (\rho) \int _{0}^{\rho} \frac{ \hat{u}_{2}(y)G(y)y^{10} }{(1-y^2)^{3} } dy,
  \end{align*}
  for some constants $c_{1}, c_{2} \in \mathbb{C}$ and for all
  $\rho \in (0,1)$.  The fact that
  $u_1\in H^{6}_\mathrm{rad}(\mathbb B^{11})$ implies $c_2=0$ as
  $\hat{u}_2$ has a ninth order pole at $\rho=0$. Therefore
  \begin{equation}\label{Eq:GenSol}
    u_{1} (\rho)  = c_{1}   \hat{u}_{1} (\rho) 
    +\hat{u}_{2} (\rho) \int _{0}^{\rho} \frac{ \hat{u}_{1}(y)G(y)y^{10} }{(1-y^2)^{3}} dy  - \hat{u}_{1} (\rho) \int _{0}^{\rho} \frac{ \hat{u}_{2}(y)G(y)y^{10} }{(1-y^2)^{3} } dy.
  \end{equation}
  The last term in Eq.~\eqref{Eq:GenSol} is smooth on $[0,1]$. To
  analyze the second term, we set
  \begin{align*}\label{def:Id}
    \mathcal{I}(\rho):=\hat{u}_{2}(\rho)\int_{0}^{\rho}\frac{F(y)}{(1-y)^{3}}dy,
  \end{align*}
  where
  \begin{align*}
    F(y):= \frac{\hat{u}_{1}(y)G(y)y^{10}}{(1+y)^3}
=\frac{y^{10}\big(2y\phi_0'(y)\phi_0''(y)+5\phi_0'(y)^2\big)}{(1+y)^3}.
  \end{align*}
  By a direct calculation we get $F''(1)\not=0$ and thus, the expansion
  of $\mathcal I(\rho)$ near $\rho=1$ contains a term of the form
  $(1-\rho)^3\log(1-\rho)$.  Consequently,
  $\mathcal I^{(4)}\notin L^2(\frac12,1)$, which is a contradiction
to $u_1\in H^6_{\mathrm{rad}}(\mathbb B^{11})$.
	
Finally we prove~\eqref{Eq:SemigroupUnstable}
  and~\eqref{Eq:SemigroupDecay}. Note
  that~\eqref{Eq:SemigroupUnstable} follows from the fact that $\la=1$
  is an eigenvalue of the operator $\mb L$ with eigenfunction $\mb g$
  and $\rg \mb P=\langle \mb g \rangle$. Next, from Corollary
  \ref{Cor:Spectrum_L} and Proposition \ref{Prop:UnifBoundRes} we
  deduce the existence of constants $D,\epsilon>0$ such that
  \[ \| \mb R_{\mb L}(1-\mb P) \|\leq D \]
  for all complex $\la$ with $\Re \la > -\epsilon$. Thus,
  \eqref{Eq:SemigroupDecay} follows from the Gearhart-Pr\"uss Theorem, see \cite{EngelNagel00}, p. 302,
  Theorem 1.11.
\end{proof}

\subsection{Estimates for the nonlinearity} 
In the next section we employ a fixed point argument to prove the
existence of decaying solutions to Eq.~\eqref{Eq:Duhamel} for small
initial data. To accomplish that, we need a Lipschitz-type estimate
for the nonlinear operator $\mb N$, see~\eqref{Eq:Nonlinearity}. We
first define
\begin{equation*}
  \mathcal{B}_{\delta}:=\left\{ \mathbf{u}\in \mathcal{H}:\|\mathbf{u}\|=\|(u_1,u_2)\|_{H^6_{\text{rad}}(\mathbb{B}^{11}) \times H^5_{\text{rad}}(\mathbb{B}^{11})}\leq \delta \right\}.
\end{equation*}
\begin{lemma}\label{Lem:Nonlinearity}
  Let $\delta>0$. For $\mathbf{u,v}\in\mathcal{B}_{\delta}$, we have
  \begin{equation*}
    \| \mathbf{N(u)-N(v)} \|\lesssim (\|\mathbf{u}\|+\|\mathbf{v}\|)\|\mathbf{u-v}\|.
  \end{equation*}
\end{lemma}
\begin{proof}
  Based on~\eqref{Eq:Nonlinearity} and~\eqref{eq:sol9}, the difference
  $N(\rho,u)-N(\rho,v)$ can be written as
  \begin{equation}
    N(\rho,u)-N(\rho,v)=\sum_{j=1}^{4}n_j(\rho^2)(u^{j+1}-v^{j+1}),
  \end{equation}
  where $n_j\in C^{\infty}[0,1]$. For $\delta>0$,
  $\mathbf{u,v}\in\mathcal{B}_{\delta}$, and due to the bilinear
  estimate
  \[ \|f_1f_2\|_{H^6_{\text{rad}}(\mathbb{B}^{11})}\lesssim
  \|f_1\|_{H^6_{\text{rad}}(\mathbb{B}^{11})}\|f_2\|_{H^6_{\text{rad}}(\mathbb{B}^{11})}, \]
  we have
  \begin{align*}\label{eq:nonlinestimate}
    \begin{split}
      \|\mathbf{N(u)-N(v)} \|&=	\|N(\cdot,u_1)-N(\cdot,v_1)\|_{H^5_{\text{rad}}(\mathbb{B}^{11})}\\
      &\leq\|N(\cdot,u_1)-N(\cdot,v_1)\|_{H^6_{\text{rad}}(\mathbb{B}^{11})}\\
      &\lesssim\sum_{j=1}^{4}\|n_j((\cdot)^2)\|_{H^6_{\text{rad}}(\mathbb{B}^{11})}
\|u_1^{j+1}-v_1^{j+1}\|_{H^{6}_{\text{rad}}(\B^{11})}
      \\
      &\lesssim\left(\|u_1\|_{H^6_{\text{rad}}(\mathbb{B}^{11})}+\|v_1\|_{H^6_{\text{rad}}(\mathbb{B}^{11})}\right)\|u_1-v_1\|_{H^6_{\text{rad}}(\mathbb{B}^{11})}\\
      &\leq(\|\mathbf{u}\|+\|\mathbf{v}\|)\|\mathbf{u-v}\|.
    \end{split}
  \end{align*}
\end{proof}

\subsection{The abstract nonlinear Cauchy problem}
In this section we treat the existence and uniqueness of solutions to
Eq.~\eqref{Eq:PerturbedProblem} for small initial data. According to
Definition \ref{Def:Solution} we study the integral equation
\begin{equation}\label{Eq:Duhamel2}
  \mathbf{\Phi}(\tau)=\mathbf{S}(\tau)\mathbf{U}(\mathbf{v},T)+\int_{0}^{\tau}\mathbf{S}(\tau-s)\mathbf{N}(\Phi(s))ds,
\end{equation}
for $\tau\geq0$ and $\mb v\in \mathcal{H}$ small. In order to employ a fixed
point argument, we introduce the necessary definitions. First, we define a
Banach space
\begin{equation}
  \mathcal X := \{ \Phi \in C([0,\infty), \mathcal H): \| \Phi  \|_{\mathcal X} := \sup_{\tau > 0} e^{\epsilon \tau} \| \Phi(\tau) \|  < \infty \},
\end{equation}
where $\epsilon$ is sufficiently small and fixed. We denote by $\mc
X_\delta$ the closed ball in $\mathcal X$ with radius $\delta$, that is,
\begin{equation}
  \mc X_{\delta} := \{\Phi \in \mc X: \| \Phi \|_{\mc X} \leq \delta \}.
\end{equation}
Finally, we define the correction term
\begin{align*}
  \mb C(\Phi, \mb u) := \mb P \left(\mb u + \int_0^{\infty} e^{-s}  \mb N(\Phi(s)) ds \right),
\end{align*}
and set
\begin{align*}
  \mb K(\Phi, \mb u)(\tau) :=\mb S(\tau)(\mb u-\mb C(\Phi, \mb u))  + \int_0^{\tau} \mb S(\tau - s)  \mb N(\Phi(s)) ds .
\end{align*}
The correction term serves the purpose of suppressing the exponential
growth of the semigroup $\mb S(\tau)$ on the unstable space. We have
the following result.

\begin{theorem} \label{Th:GlobalEx_ModEq} There exist constants
  $\delta, C>0$ such that for every $\mb u \in \mathcal{H}$ which
  satisfies $\| \mb u \|\leq \frac{\delta}{C}$, there exists a unique
  $\Phi_{\mb u}\in \mathcal{X}_{\delta}$ such that
  \begin{equation}\label{Eq:DuhamelK}
    \Phi_{\mb u}=\mb K(\Phi_{\mb u},\mb u).
  \end{equation}
  In addition, the solution $\Phi_{\mb u}$ is unique in the whole space
  $\mathcal{X}$ and the solution map $\mb u \mapsto \Phi_{\mb u}$ is
  Lipschitz continuous.
\end{theorem}
\noindent The proof coincides with the one of Theorem 3.7 in
\cite{ChaDonnGlo17}.

We now study the initial data $\mb U(\mb v,T)$, see
~\eqref{Eq:InitialDataU}, and prove its continuity in $T$ near
$T_0$. For that reason we define
\[\mc H^R := H_{\text{rad}}^{6}\times
H_{\text{rad}}^{5}(\mathbb{B}^{11}_R),\] with the induced norm
\[ \| \mb w \|^2_{\mc H^R} = \| w_1(|\cdot|) \|^2_{H^{6}(\B^{11}_R)} +
\| w_2(|\cdot|) \|^2_{H^{5}(\B^{11}_R)}.\]
\begin{lemma}\label{Le:InitialData}
  Fix $T_0>0$. Let $|\cdot|^{-1}\mb v \in \mc H^{T_0+\delta}$ for
  $\delta$ positive and sufficiently small. Then the map
  \[ T \mapsto \mb U(\mb v, T): [T_0-\delta, T_0 + \delta ] \to \mc
  H \]
  is continuous. Furthermore, for all
  $T \in [T_0- \delta, T_0 + \delta]$,
  \[ \||\cdot|^{-1} \mb v \|_{\mc H^{T_0 + \delta}} \leq \delta
  \implies \|\mb U(\mb v, T) \| \lesssim \delta. \]
\end{lemma}
\begin{proof}
  We prove the result for $T_0=1$ only, as the general case is treated
  similarly. Assume $|\cdot|^{-1}\mb v\in \mathcal{H}^{1+\delta}$ for
  $\delta$ positive but less than $\frac{T_0}{2}=\frac{1}{2}$. We
  first introduce some auxiliary facts. Namely, by scaling we see that
  for $f\in H^{6}_{\text{rad}}(\B^{11}_{1+\delta})$ and
  $T\in[1-\delta,1+\delta]$
  \[ \| f(|T\cdot|) \|_{H^{6}(\B^{11}_{1})}\lesssim \| f(|\cdot|)
  \|_{H^{6}(\B^{11}_{1+\delta})}.  \]
  Furthermore, from the density of
  $C^{\infty}_{\text{even}}[0,1+\delta]$ in
  $H^{6}_{\text{rad}}(\B^{11}_{1+\delta})$ we conclude that given
  $\ve>0$, there exists a
  $\tilde{v}_1\in C^{\infty}_{\text{even}}[0,1+\delta]$ such that
  $\| |\cdot|^{-1} v_1(|\cdot|)-\tilde{v}_1(|\cdot|)
  \|_{H^{6}_{\text{rad}}(\B^{11}_{1+\delta})}<\ve$.
  Also, the functions $\frac{1}{T\rho}\phi_0(T\rho)$ and
  $\tilde{v}(T\rho)$ are smooth on $[0,1]$ for
  $T\in[1-\delta,1+\delta]$.  Therefore,
  \begin{equation}\label{Eq:Phi0limit}
    \lim_{T\rightarrow \widetilde{T}}\| |T\cdot|^{-1}\big (\phi_0(|T\cdot|)-\phi_0(|\widetilde{T}\cdot|) \|_{H^{6}(\B^{11})}+ \|\tilde{v}_1(|T\cdot|)-\tilde{v}_1(|\widetilde{T}\cdot|) \|_{H^{6}(\B^{11})}=0.
  \end{equation} 
  Using these facts, we prove the continuity of the first component of the
  map $T \rightarrow \mb U(\mb v,T)$. Namely, given $\ve>0$, there
  exists a $\tilde{v}_1\in C^{\infty}_{\text{even}}[0,1+\delta]$ such
  that for $T,\widetilde{T}\in[1-\delta,1+\delta]$ we have
  \begin{align*}
    \|  [\mb U&(\mb v,T)]_1-[\mb U(\mb v,\widetilde{T})]_1 \|_{H^6(\B^{11})}\\
    =&\big\| |\cdot|^{-1}\phi_0(|T\cdot|) + |\cdot|^{-1}v_1(|T\cdot|) - |\cdot|^{-1}\phi_0(|\widetilde{T}\cdot|) - |\cdot|^{-1}v_1(|\widetilde{T}\cdot|) \big\|_{H^6(\B^{11})}\\
    \lesssim &\big\| |\cdot|^{-1}\big( \phi_0(|T\cdot|) - \phi_0(|\widetilde{T}\cdot|)\big) \big\|_{H^6(\B^{11})}  +
               \big\| |T\cdot|^{-1} v_1(|T\cdot|) - \tilde{v}_1(|T\cdot|) \big\|_{H^6(\B^{11})}\\
              &+\big\| \tilde{v}_1(|T\cdot|) - \tilde{v}_1(|\widetilde{T}\cdot|)\big\|_{H^6(\B^{11})} +
		\big\| |T\cdot|^{-1} \tilde{v}_1(|\widetilde{T}\cdot|) - v_1(|\widetilde{T}\cdot|) \big\|_{H^6(\B^{11})}\\
    \lesssim& \big\| |T\cdot|^{-1}\big( \phi_0(|T\cdot|) - \phi_0(|\widetilde{T}\cdot|)\big) \big\|_{H^6(\B^{11})}  +
              \big\| |\cdot|^{-1} v_1(|\cdot|) - \tilde{v}_1(|\cdot|) \big\|_{H^6(\B_{1+\delta}^{11})} \\
              &+\big\| \tilde{v}_1(|T\cdot|) - \tilde{v}_1(|\widetilde{T}\cdot|) \big\|_{H^6(\B^{11})}\\
    \leq& \big\| |T\cdot|^{-1}\big (\phi_0(|T\cdot|)-\phi_0(|\widetilde{T}\cdot|) \big\|_{H^{6}(\B^{11})}+ \big\|\tilde{v}_1(|T\cdot|)-\tilde{v}_1(|\widetilde{T}\cdot|) \big\|_{H^{6}(\B^{11})}+\ve,
  \end{align*}
  This together with~\eqref{Eq:Phi0limit} implies that
  $[\mb U(\mb v,T)]_1$ is continuous. The second component is treated analogously.
 Now, given
  $\big\||\cdot|^{-1}\mb v\big\|_{\mathcal{H}^{1+\delta}}\leq\delta$
  and $T\in[1-\delta,1+\delta]$, we have
  \begin{align*}
    \|  [\mb U(\mb v,T)]_1\|_{H^6(\B^{11})} &=\big\| |\cdot|^{-1}\phi_0(|T\cdot|)-|\cdot|^{-1}\phi_0(|\cdot|)+|\cdot|^{-1}v_1(|T\cdot|) \big\|_{H^6(\B^{11})}\\
                                            &\lesssim |T-1|+\big\||\cdot|^{-1}v_1\big\|_{H^6(\B^{11}_{1+\delta})} \lesssim \delta.
  \end{align*}
  We obtain a similar estimate for the second component and finally
  deduce that
  \[ \|\mb U(\mb v,T)\|\lesssim \delta. \]
\end{proof}
As already mentioned, the unstable eigenvalue $\la=1$ is present due
to the freedom of choice of the parameter $T$, and is therefore not
considered a ``real'' instability of the linear problem. The
following theorem is the precise version of this statement. Namely, for
a given $T_0$ and small enough initial data $\mb v$, there exists a
$T_{\mb v}$ close to $T_0$ that makes the correction term
$\mb C(\Phi_{\mb U(\mb v,T_{\mb v})},\mb U(\mb v,T_{\mb v}))$ vanish. This in turn
allows for proving the existence and uniqueness of an exponentially
decaying solution to Eq.~\eqref{Eq:Duhamel2}.
\begin{theorem}\label{Thm:GlobalExistence}
  Fix $T_0>0$. Then there exist $\delta,M > 0$ such that for any
  $\mb v$ that satisfies
  \[ \big\| |\cdot|^{-1}\mb
  v\big\|_{\mathcal{H}^{T_0+\delta}}\leq\frac{\delta}{M} \]
  there exists a $T \in [T_0-\delta,T_0+\delta]$ and a function
  $\Phi \in \mathcal{X}_{\delta}$ which satisfies
  \begin{align}\label{Eq:DuhamelFinal}
    \Phi(\tau)=\mb S(\tau)\mb U(\mb v,T)+\int_0^{\tau}\mb S(\tau-s)\mb N(\Phi(s))ds
  \end{align}
  for all $\tau > 0$.  Moreover, $\Phi$ is the unique solution of this
  equation in $C([0,\infty),\mc H)$.
\end{theorem}
\begin{proof}
  Let $T_0>0$ be fixed. We first prove that for any $T$ in a small
  neighborhood of $T_0$ and small enough initial data $\mb v$ there
  exists a unique solution to Eq.~\eqref{Eq:DuhamelK} for
  $\mb u=\mb U(\mb v,T)$.  From Lemma \ref{Le:InitialData} we deduce
  the existence of sufficiently small $\delta$ and sufficiently large
  $M>0$ so that for every $T\in[T_0-\delta,T_0+\delta]$,
  $\big\| |\cdot|^{-1} \mb v \big\|_{\mathcal{H}^{T_{0}+\delta} } \leq
  \frac{\delta}{M}$
  implies $\|\mb U(\mb v, T)\|_{\mathcal{H}}\leq\frac{\delta}{C}$ for
  a large enough $C>0$. Via Theorem \ref{Th:GlobalEx_ModEq} this
  yields the unique solution to Eq.~\eqref{Eq:DuhamelK} for every $T$
  in the designated range.  It remains to show that for small enough
  $\mathbf v$, there exists a particular
  $T_{\mathbf{v}} \in [T_{0}-\delta,T_{0}+\delta]$ that makes the
  correction term vanish, i.e.,
  $\mathbf C(\Phi_{\mathbf U(\mathbf v,T_{\mathbf
    v})},\mathbf U(\mathbf v,T_{\mathbf
    v}))=0$.
  Since $\mathbf C$ has values in
  $\rg \mathbf P=\langle\mathbf g\rangle$, the latter is equivalent to
  the existence of a $T_{\mb v} \in [T_0-\delta,T_0+\delta]$ such that
  \begin{align}
    \label{Eq:Tv}
    \langle \mathbf{C} \left( \Phi _{\mathbf{U}( \mathbf{v},T_{\mathbf{v}} )},\mathbf{U} \left( \mathbf{v},T_{\mathbf{v}} \right) \right),\mathbf{g} \rangle_{\mathcal H} = 0.
  \end{align}
By definition, we have
  \[ \partial_T
  \begin{bmatrix}
    \frac{1}{\rho}\phi_0(\frac{T}{T_0}\rho) \\
    \frac{T^2}{T_0^2}\phi_0'(\frac{T}{T_0}\rho)
  \end{bmatrix} \Bigg|_{T=T_0}=\frac{\mb g(\rho)}{T_0} \]
and this yields the expansion
  \[ \langle \mathbf{C} \left( \Phi _{\mathbf{U}( \mathbf{v},T_{\mathbf{v}} )},\mb U (\mb{v},T )
  \right),\mb{g} \rangle_{\mathcal H}=\frac{\|\mb g\|^2}{T_0}(T-T_0)
  +O((T-T_0)^2)+O(\tfrac{\delta}{M}T^0)+O(\delta^2T^0).
  \]
A simple fixed point argument now proves \eqref{Eq:Tv}, see
  \cite{DonSch16}, Theorem 4.15 for full details.
\end{proof}

\begin{proof}[Proof of Theorem \ref{Th:Main}]
  Fix $T_0>0$ and assume the radial initial data $u [0]$ satisfy
  \begin{align*}
    \left  \| |\cdot|^{-1} \Big( u [0] -u^{T_{0}}[0]  \Big) \right \|_{ H^6 (\mathbb{B}_{T_{0}+\delta}^{11}) \times H^5 (\mathbb{B}_{T_{0}+\delta}^{11}  )} \leq \frac{\delta}{M_0^2}
  \end{align*}
  with $\delta, M_0>0$ to be chosen later. We set
  $\mathbf v:=u[0]-u^{T_0}[0]$, see Section
  \ref{Sec:PerturbedProblem}.  Then we have
  \begin{align*}
    \left \| |\cdot|^{-1} \mathbf{v} \right \|_{\mathcal{H}^{T_{0}+\delta}}  = \left \| |\cdot|^{-1} \Big( u [0] -u^{T_0}[0] \Big) 
    \right \|_{\mathcal{H}^{T_{0}+\delta}} \leq \frac{\delta}{M_0^2}.
  \end{align*}
  Now, upon choosing $\delta>0$ sufficiently small and $M_0>0$
  sufficiently large, Theorem \ref{Thm:GlobalExistence} yields a
  $T \in [T_0-\frac{\delta}{M_0},T_0+\frac{\delta}{M_0}]\subset
  [1-\delta,1+\delta]$ such that there exists a unique solution
  $\Phi=(\varphi_1,\varphi_2) \in \mathcal{X}$ to Eq.~\eqref{Eq:DuhamelFinal} with
  $\| \Phi (\tau) \| \leq \frac{\delta}{M_0} e^{-2\epsilon \tau}$ for all
  $\tau\geq 0$ and some $\epsilon>0$.  Therefore, by construction,
  \[ u(t,r)=u^T(t,r)+\frac{r}{T-t}\,\vp_1\left
    (\log\frac{T}{T-t},\frac{r}{T-t}\right ) \]
  solves the original wave maps equation \eqref{Eq:CorWMd9}. Moreover,
  \[ \partial_t u(t,r)=\partial_t u^T(t,r)+\frac{r}{(T-t)^2}\,\vp_2
  \left (\log\frac{T}{T-t},\frac{r}{T-t}\right ). \] Therefore,
  \begin{align*}
    (T-t)^{k-\frac{9}{2}}&\left \||\cdot|^{-1}\left (u(t,|\cdot|)-u^T(t,|\cdot|)\right )
                           \right \|_{\dot H^k(\mathbb B^{11}_{T-t})}  \\
                         &=(T-t)^{k-\frac{11}{2}}\left \|\vp_1\left (\log\frac{T}{T-t},\frac{|\cdot|}{T-t}\right ) \right \|_{\dot H^k(\mathbb B^{11}_{T-t})} \\
                         &=\left \|\vp_1\left (\log\frac{T}{T-t},|\cdot|\right ) \right \|_{\dot H^k(\mathbb B^{11})}
                           \leq \left \|\Phi\left (\log\frac{T}{T-t}\right )\right \|_{\mathcal{H}} \\
                         &\leq \tfrac{\delta}{M_0} (T-t)^{2\epsilon}
  \end{align*}
  for all $t\in [0,T)$ and any integer $0\leq k \leq 6$.  Furthermore,
  \begin{align*}
    (T-t)^{l-\frac{7}{2}}&\left \||\cdot|^{-1}\left (\partial_t u(t,|\cdot|)-\partial_t u^T(t,|\cdot|)\right )\right \|_{\dot H^l(\B^{11}_{T-t})} \\
                         &=(T-t)^{l-\frac{11}{2}}\left \|\vp_2\left (\log\frac{T}{T-t},\frac{|\cdot|}{T-t}\right ) \right \|_{\dot H^l(\B^{11}_{T-t})} \\
                         &=\left \|\vp_2\left (\log\frac{T}{T-t},|\cdot|\right ) \right \|_{\dot H^l(\B^{11})}
                           \leq \left \|\Phi\left (\log\frac{T}{T-t}\right )\right \|_{\mathcal{H}} \\
                         &\leq \tfrac{\delta}{M_0} (T-t)^{2\epsilon}
  \end{align*}
  for all $l=0,1,\dots,5$.  Finally, by Sobolev embedding we infer
  \begin{align*}
    \|u(t,\cdot)-u^T(t,\cdot)\|_{L^{\infty}(0,T-t)} &\leq (T-t)\big\||\cdot|^{-1}\left( u(t,|\cdot|)-u^T(t,|\cdot|)\right)\big\|_{L^{\infty}(0,T-t)} \\ &\lesssim (T-t)\big\||\cdot|^{-1}\left( u(t,|\cdot|)-u^T(t,|\cdot|)\right)  \big\|_{H^{\frac{11}{2}+\epsilon}(\B^{11}_{T-t})}\\
                                                    &\lesssim \tfrac{\delta}{M_0} (T-t)^{\epsilon}
  \end{align*}
and this finishes the proof by setting $M:=M_0^2$.
\end{proof}
\begin{remark}
  Based on \cite{DonSch16,ChaDonnGlo17}, the analogue of
  Theorem \ref{Th:Main} in any odd dimension $d\geq11$ follows from the
  mode stability of the solution $u^T$. This will be addressed in a
  forthcoming publication.
\end{remark}
\appendix

\section{Proof of Proposition
  \ref{Prop:NegCurv}}\label{Sec:ProofNegCurv}
\noindent A straightforward computation shows that all
 sectional curvatures of the manifold $N^d$ are given by
either
\begin{equation}\label{eq:curvatures}
  (i)\enskip \frac{-g''(u)}{g(u)} \quad \text{or} \quad (ii)\enskip \frac{1-g'(u)^2}{g(u)^2}.
\end{equation} 
We first show that the two expressions above are negative provided $d\geq8$ and $u\in I:=[0,\phi_0(1)]$.
For convenience we let $d=e+8$. We now have
\begin{equation}\label{eq:sec_curv}
  \frac{g''(u)}{g(u)}=\frac{6(23e+14)^2u^6-63(23e+14)u^4-2(115e+21)u^2+21}{[(23e+14)u^4-7u^2-1]^2}.
\end{equation}
Denote the numerator in the above expression by $N(e,u)$. To show that
the first quantity in~\eqref{eq:curvatures} is negative it suffices to 
prove that $N(e,u)>0$ for $(e,u)\in [0,\infty)\times I$. To that end,
it is enough to show that for any fixed $e\geq 0$ the following
inequalities hold
\begin{equation}\label{eq:claims0}
  (i)\enskip N(e,0)>0, \quad (ii)\enskip N(e,\phi_0(1))>0\quad \text{and} \quad (iii)\enskip \partial^2_uN(e,u)<0\enskip \text{for } u\in I.
\end{equation}
We start by proving the third claim above. Note that it is enough to
show that
\begin{equation}\label{eq:claims}
  (i)\enskip \partial^2_uN(e,0)<0, \quad \text{and} \quad (ii)\enskip\partial^3_uN(e,u)\leq0 \enskip \text{for}\enskip u\in I.
\end{equation} To establish~\eqref{eq:claims} we need the following
\begin{align}
  \partial^2_uN(e,u)&=4[45(23e+14)^2u^4-189(23e+14)u^2-115e-21],\label{eq:par2}\\
  \partial^3_uN(e,u)&=72(23e+14)u[10(23e+14)u^2-21]\enskip \text{and}\label{eq:par3}\\
  \partial^5_uN(e,u)&=4320u(23e+14)^2.\label{eq:par5}
\end{align}
Equation~\eqref{eq:par2} gives $\partial^2_uN(e,0)=-4(115e+21)$ and
the first claim in~\eqref{eq:claims} follows. From ~\eqref{eq:par5} we
see that $\partial^3_uN(e,u)$ is convex for $u\in I$. Therefore, since
$\partial^3_uN(e,0)=0$ it is enough to show that
\begin{equation}\label{eq:D3N}
  \partial^3_uN(e,\phi_0(1))\leq0,
\end{equation} for the second claim in~\eqref{eq:claims} to hold. To establish this inequality, we first use definition~\eqref{Eq:Sol} to compute
\begin{equation*}\label{eq:u0}
  \phi_0(1)=\left(\frac{2}{\sqrt{(e+7)(46e^2+445e+567)}-7(e+7)}\right)^{\frac{1}{2}}.
\end{equation*}
Now, according to~\eqref{eq:par3}, it is enough to prove that
$10(23e+14)\phi_0(1)^2-21<0$ for \eqref{eq:D3N} to hold. This
inequality is equivalent to $441e^2-925e+1316>0$, which clearly holds
for all $e\geq0$. This concludes the proof of the third claim in
\eqref{eq:claims0}. Since the first claim in~\eqref{eq:claims0} is
obviously true it is left to prove that $N(e,\phi_0(1))>0$. To that
end we first compute
\begin{equation}\label{eq:N0}
  N(e,\phi_0(1))=\frac{2(P(e)\sqrt{Q(e)}-R(e))}{[\sqrt{Q(e)}-7(e+7)]^3},
\end{equation}
where
\begin{align*}
  P(e)&=7(69e^3+1831e^2+11500e+17094),\\
  Q(e)&=(e+7)(46e^2+445e+567)\enskip \text{and}\\
  R(e)&=20723e^4+433338e^3+3077307e^2+8566502e+7537866.
\end{align*}
The denominator in~\eqref{eq:N0} is positive if and only of
$Q(e)^2-49(e+7)^2>0$. This is equivalent to $2(e+8)(e+7)(23e+14)>0$,
which is manifestly true for $e\geq 0$. The numerator in~\eqref{eq:N0}
is positive if and only if $P(e)^2Q(e)-R(e)^2>0$, which is equivalent
to $2(23e+14)^2S(e)>0$ where
\begin{align*}\label{eq:S}
  S(e)=10143e^7+289189e^6&+2979735e^5+12402439e^4\\&+11046366e^3-30567884e^2+15651132e+22614480.
\end{align*}
The positivity of $S(e)$ is easily shown; for example we have
\[
12402439e^4+22614480>30567884e^2.
\]
The positivity of $N(e,\phi_0(1))$ follows.
	
Now we turn to proving that the second expression in~\eqref{eq:curvatures} is negative for $d\geq8$ and $u\in I$. Since $g''(u)/g(u)$ is positive for $u\in I$ and $g(u)>0$ for small positive values of $u$, we conclude that both $g''$ and $g$ are positive on $(0,\phi_0(1)]$.
Consequently
\[ g'(u)-1=g'(u)-g'(0)=\int_0^u g''(t)dt>0 \quad \text{for} \quad u\in(0,\phi_0(1)].\]
Hence $g'(u)^2-1>0$ and therefore $$\frac{1-g'(u)^2}{g(u)^2}<0$$ for $u\in(0,\phi_0(1)]$.
Additionally, by direct computation we see that $$\frac{1-g'(0)^2}{g(0)^2}=-21<0.$$
Finally, for each $d\geq 8$ we infer the existence of $\ve>0$ for which both expressions in~\eqref{eq:curvatures} are negative provided $|u|<\phi_0(1)+\ve$. For $|u|\geq\phi_0(1)+\ve$, the function $g(u)$ can be easily modified so that it satisfies~\eqref{eq:assumg} and both expressions in~\eqref{eq:curvatures} remain negative.

\section{Estimate for $\delta_7$}
\begin{proposition}\label{Prop:AppendixEstimate}
  For $\delta_7$ defined in Eq.~\eqref{Def:Delta} and $\la\in\Hb$ we have
  \begin{equation}\label{Eq:delta_estimate}
    |\delta_7(\la)|\leq\frac{1}{3}.
  \end{equation}
\end{proposition}
\begin{proof}
  Following the proof of Lemma 4.3 in \cite{CDG17} we show that for
  $r_7$ and $(\tilde{r}_7)^{-1}$ are analytic in $\Hb$. This implies
  that $\delta_7$ is also analytic there. Furthermore, being a
  rational function, $\delta_7$ is evidently polynomially bounded in
  $\Hb$. Therefore, according to the Phragm\'{e}n-Lindel\"{o}f
  principle\footnote{We use the sectorial formulation of this
    principle, see, for example, \cite{Titchmarsh58}, p. 177.}, it
  suffices to prove that~\eqref{Eq:delta_estimate} holds on the
  imaginary line, i.e.
  \begin{equation}\label{Eq:Delta7estimate}
    |\delta_7(is)|^2\leq\frac{1}{9} \quad\text{for }s\in\mathbb{R}.
  \end{equation} 
  Note that the function $s\mapsto|\delta_7(is)|^2$ is even. It is
  therefore enough to prove~\eqref{Eq:Delta7estimate} for nonnegative
  $s$ only. We show that for $t\geq 0$,
  \begin{equation}\label{Eq:tEstimates}
    \left|\delta_7\left(\frac{4t}{t+1}i\right)\right|^2\leq\frac{1}{9} \quad \text{and} \quad \left|\delta_7\big( (t+4)i \big) \right|^2\leq\frac{1}{9}.
  \end{equation}
  The first estimate above proves~\eqref{Eq:Delta7estimate} for
  $s\in[0,4)$, while the second one covers the complementary interval
  $[4,\infty).$ We prove both estimates in~\eqref{Eq:tEstimates} in
  the same way and therefore illustrate the proof of the second one
  only.  Note that
 $$\left|\delta_7\big( (t+4)i \big) \right|^2=\frac{Q_1(t)}{Q_2(t)}$$
 where $Q_j(t)\in\mathbb{Z}[t]$, $\deg Q_j=32$ and $Q_2$ has all
 positive coefficients.  Therefore,
 $|\delta_7\big( (t+4)i \big) |^2\leq\frac{1}{9}$ is equivalent to
 $Q_2-9Q_1\geq0$ and a direct calculation shows that the
 polynomial $Q_2-9Q_1$ has manifestly positive coefficients.
\end{proof}

\bibliography{references1} \bibliographystyle{plain}

\end{document}